\documentclass[10pt,a4paper]{article}

\usepackage{hyperref} 

\usepackage[utf8]{inputenc}
\usepackage{cite}

\usepackage{amssymb,amsmath,amsthm,mathrsfs}
\usepackage[english]{babel}

\usepackage{subfigure}

\usepackage{todonotes}
\usepackage{verbatim} 
\usepackage{enumerate}
\usepackage{mathtools}

\usepackage{a4wide}
\usepackage{multirow}
\usepackage[footnotesize,bf,centerlast]{caption}
\usepackage[small,euler-digits,icomma,OT1,T1]{eulervm}

\usepackage{xcolor,colortbl}
\definecolor{Gray}{gray}{0.80}
\definecolor{LightGray}{gray}{0.90}

\usepackage[normalem]{ulem}

\usepackage{bbm}

\newcommand{\cA}{\mathcal{A}}

\newcommand{\cC}{\mathcal{C}}
\newcommand{\cD}{\mathcal{D}}

\newcommand{\cG}{\mathcal{G}}
\newcommand{\cH}{\mathcal{H}}

\newcommand{\cL}{\mathcal{L}}

\newcommand{\cP}{\mathcal{P}}

\newcommand{\fS}{\mathfrak{S}}

\newcommand{\bF}{\mathbb{F}}

\newcommand{\bN}{\mathbb{N}}

\newcommand{\bR}{\mathbb{R}}

\newcommand{\PR}{\mathbb{P}}

\newcommand{\dd}{ \mathrm{d}}

\DeclareMathOperator{\atan2}{atan2} 
\DeclareMathOperator{\LIM}{LIM}

\renewcommand{\epsilon}{\varepsilon}

\newcommand{\vn}[1]{\left| \! \left| #1\right| \! \right|}

\newcommand{\ip}[2]{\langle #1,#2\rangle}

\numberwithin{equation}{section}

\newtheorem{theorem}{Theorem}[section]
\newtheorem{lemma}[theorem]{Lemma}
\newtheorem{proposition}[theorem]{Proposition}

\theoremstyle{definition}
\newtheorem{definition}[theorem]{Definition}

\newtheorem{remark}[theorem]{Remark}
\newtheorem*{remark*}{Remark}

\newtheorem{assumption}[theorem]{Assumption}
\newtheorem*{example}{Example}

\newcommand{\Ftodo}[1]{\todo[size=\scriptsize,color=green!30]{#1}}
\newcommand{\Ftodoline}[1]{\todo[inline,color=green!30]{#1}}
\newcommand{\Fra}[1]{{\color{blue} #1}}

\title{Path-space moderate deviations for a class of Curie-Weiss models with dissipation}

\author{Francesca Collet\footnotemark[1]\thanks{Delft Institute of Applied Mathematics, Delft University of Technology, van Mourik Broekmanweg 6, 2628 XE Delft (The Netherlands). \emph{E-mail addresses}: $\{$f.collet-1, r.c.kraaij$\}$@tudelft.nl} \and Richard C. Kraaij\footnotemark[1]}

\date{\today}

\begin{document}

\maketitle

\begin{abstract}
\noindent We modify the Glauber dynamics of the Curie-Weiss model with dissipation in \cite{dfr} by considering arbitrary transition rates and we analyze the phase-portrait as well as the dynamics of moderate fluctuations for macroscopic observables. We obtain path-space moderate deviation principles via a general analytic approach based on the convergence of non-linear generators and uniqueness of viscosity solutions for associated Hamilton-Jacobi equations. The moderate asymptotics depend crucially on the phase we are considering and, moreover, their behavior may be influenced by the choice of the rates.\\

\end{abstract}

\section{Introduction}

Examples of self-organization leading to collective phenomena in natural and social sciences are easily encountered. Schools of fishes, flocks of birds, applauding audiences, firings of neuron assemblies, pacemaker cell beats,\dots are groups of units able to organize themselves allowing coherence to arise starting from incoherent configurations \cite{PiRoKu03,Scw07}. In particular, emergence of self-sustained oscillations is among the most commonly observed ways of self-organization in ecology \cite{Tur03}, neuroscience \cite{ErTe10,LiG-ONeS-G04} and socioeconomics \cite{WeHa12}.

Various stylized models have been proposed to unveil possible universal mechanisms that enhance collective periodic behaviours. The interplay between a cooperative interaction potential and the noise seems to be crucial and, moreover, a reversibility-breaking mechanism is needed, as stochastic reversibility is actually in contrast with rhythms \cite{BeGiPa10,GiPo15}. In recent years, from a modeling viewpoint, great attention has been turned to mean-field interacting particle systems, due to their analytical tractability \cite{CoDaPFo15,CoFoTo16,DiLo17,LuPo,ToHeFa12}. We would like to mention that cyclic patterns can be produced also in a mean-field game theoretical setting \cite{DaPSaTo13, DaPSaTo}. 

We are interested in an irreversible modification of the standard Curie-Weiss model. In \cite{dfr} the authors considered Glauber dynamics for a Curie-Weiss model in which the interaction potential is subject to a dissipative and noisy stochastic evolution. They showed that, in the thermodynamic limit, for sufficiently strong interaction and zero (or sufficiently small) noise, the magnetization of the system exhibits self-sustained oscillations, i.e. it shows a time periodic behaviour despite of the fact that no periodic force is applied. More recently, fluctuations on the level of a path-space (standard and non-standard) central limit theorem for the noiseless version of the same model were studied in \cite{DaPTo18}.

Our aim is to continue the analysis of fluctuations by characterizing their dynamical features whenever looking for moderate size deviations from the average value. 

In the system we are considering each spin-flip $\sigma_i \to -\sigma_i$ occurs with positive intensity of the form $\Gamma(-\sigma_i \zeta_n)$. The process $\zeta_n$ is the effective potential felt by the $i$-th particle and it is damped in time according to the equation $d\zeta_n = -\alpha \zeta_n dt + \beta dm_n$ ($\alpha, \beta > 0$), where $m_n = n^{-1} \sum_{i=1}^n \sigma_i$ is the empirical average of the spins.  Compared to the original version of the dissipative Curie-Weiss model \cite{dfr}, we do not have any external noise in the evolution of $\zeta_n$ (as in \cite{DaPTo18}); on the other hand, we work with arbitrary transition rates rather than sticking on the case $\Gamma(x) = 1 + \tanh(x)$. 

Depending on the choice of the function $\Gamma$, the phase structure of the infinite volume system may become very rich. In addition to a scenario where the rise of oscillations occurs locally around the fixed point via Hopf bifurcation (as in \cite{dfr}), we may obtain a phase diagram in which a tri-critical point exists and the stable limit cycle may originate from a global, rather than local, bifurcation (\emph{saddle-node} bifurcation of limit cycles). To our knowledge the possibility of emergence of a periodic orbit from a non-local bifurcation was not pointed out before for this class of models.

We want to determine how microscopic observables moderately fluctuate around the stationary solution of the macroscopic dynamics in the various regimes by deriving path-space moderate deviation principles. It is worthy to mention that the system we are considering is readily tractable since, due to the mean-field nature of the interaction, the analysis leads to the study of the evolution of a two-dimensional order parameter: $(m_n,\zeta_n)$. It then suffices to characterize the asymptotics of the latter. 

We recall moreover that a moderate deviation principle is technically a large deviation principle and consists in a refinement of a (standard or non-standard) central limit theorem, in the sense that it characterizes the exponential decay of deviations from the average on a smaller scale. We apply the generator convergence approach to large deviations by \cite{FK06} to characterize the most likely behavior for the trajectories of the fluctuations of $(m_n,\zeta_n)$. 

Our findings highlight the following distinctive aspects:

\begin{itemize}
\item The moderate asymptotics depend crucially on the phase we are considering. The physical phase transition is reflected at this level via a sudden change in the speed and rate function of the moderate deviation principle. In particular, fluctuations are Gaussian-like in the subcritical regime, while they are not at criticalities. 
\item In the subcritical regime, the processes $m_n$ and $\zeta_n$ evolve on the same time-scale and we characterize deviations from the average of the pair $(m_n,\zeta_n)$. For the proof we will refer to the large deviation principle in \cite[App.~A]{CoKr17}. On the contrary, at criticality, in analogy with \cite{DaPTo18}, a suitable change of coordinates leads to a slow-fast dynamical description: in the natural  time scale, the fast variable equilibrates quickly and the limiting behavior of the slow one can be determined after averaging. Corresponding to this observation, we need to prove a path-space large deviation principle for a projected process, in other words for the slow component only. The projection on a one-dimensional subspace relies on the synergy between the convergence of the Hamiltonians \cite{FK06} and the perturbation theory for Markov processes \cite{PaStVa77}. 
\end{itemize}
The paper is organized as follows. In Section~\ref{sect:model_and_results} we introduce the family of models we are considering and we give the main results. Most of the proofs are postponed to Section~\ref{sect:proofs}. In Section~\ref{section:comparison_principle_singular_hamiltonian} we show validity of the comparison principle for a class of singular Hamiltonians. It is a key step to deduce our statements. Appendix~\ref{appendix:large_deviations_for_projected_processes} contains the mathematical tools needed to derive our large deviation principles via solving a class of associated Hamilton-Jacobi equations and it is included to make the paper as much self-contained as possible.
Appendix~\ref{appendix:proof_thm_phase_diagram} is devoted to proving the phase diagram for the macroscopic dynamics. 

\section{Model and main results}\label{sect:model_and_results}

\subsection{Notation and definitions}

Before we consider the main contents of the paper, we introduce some notation. We start with the definition of good rate-function and of large deviation principle for a sequence of random variables. 
	
\begin{definition}
Let $\{X_n\}_{n \geq 1}$ be a sequence of random variables on a Polish space $\mathcal{X}$. Furthermore, consider a function $I : \mathcal{X} \rightarrow [0,\infty]$ and a sequence $\{r_n\}_{n \geq 1}$ of positive numbers such that $r_n \rightarrow \infty$. We say that
\begin{itemize}
\item  
the function $I$ is a \textit{good rate-function} if the set $\{x \, | \, I(x) \leq c\}$ is compact for every $c \geq 0$.
\item 
the sequence $\{X_n\}_{n\geq 1}$ is \textit{exponentially tight} at speed $r_n$ if, for every $a \geq 0$, there exists a compact set $K_a \subseteq \mathcal{X}$ such that $\limsup_n r_n^{-1} \log \, \PR[X_n \notin K_a] \leq - a$.
\item 
the sequence $\{X_n\}_{n\geq 1}$ satisfies the \textit{large deviation principle} with speed $r_n$ and good rate-function $I$, denoted by 
\begin{equation*}
\PR[X_n \approx a] \asymp e^{-r_n I(a)},
\end{equation*}
if, for every closed set $A \subseteq \mathcal{X}$, we have 
\begin{equation*}
\limsup_{n \rightarrow \infty} \, r_n^{-1} \log \PR[X_n \in A] \leq - \inf_{x \in A} I(x),
\end{equation*}
and, for every open set $U \subseteq \mathcal{X}$, 
\begin{equation*}
\liminf_{n \rightarrow \infty} \, r_n^{-1} \log \PR[X_n \in U] \geq - \inf_{x \in U} I(x).
\end{equation*}
\end{itemize}
\end{definition}
	
Throughout the whole paper $\cA\cC$ will denote the set of absolutely continuous curves in $\bR^d$. For the sake of completeness, we recall the definition of absolute continuity.

\begin{definition} 
A curve $\gamma: [0,T] \to \mathbb{R}^d$ is absolutely continuous if there exists a function $g \in L^1([0,T],\bR^d)$ such that for $t \in [0,T]$ we have $\gamma(t) = \gamma(0) + \int_0^t g(s) \dd s$. We write $g = \dot{\gamma}$.\\
A curve $\gamma: \bR^+ \to \mathbb{R}^d$ is absolutely continuous if the restriction to $[0,T]$ is absolutely continuous for every $T \geq 0$. 	
\end{definition}

To conclude we fix notation for a collection of function-spaces. Let $k \geq 1$ and $E$ a closed subset of $\mathbb{R}^d$. We will denote by $C_c^k(E)$ the set of functions that are constant outside some compact set in the interior of $E$ and are $k$ times continuously differentiable on a neighborhood of $E$ in $\mathbb{R}^d$. Finally, we define $C_c^\infty(E) := \cap_k C_c^k(E)$.

\subsection{Description of the model}\label{ss:dissipative:CW}

Let $\sigma = \left( \sigma_i \right)_{i=1}^n \in \{-1,+1\}^n$ be a configuration of $n$ spins.  The stochastic process $\{\sigma(t)\}_{t \geq 0}$ is described as follows. For $\sigma \in \{-1,+1\}^n$, let us define $\sigma^j$ the configuration obtained from $\sigma$ by flipping the $j$-th spin. The spins will be assumed to evolve with Glauber one spin-flip dynamics: at any time $t$, the system may experience a transition $\sigma \longrightarrow \sigma^j$ at rate $\Gamma (-\sigma_j \zeta_n)$, where $\Gamma: \mathbb{R} \to \mathbb{R}$ is a {\em positive and increasing map} and $\{\zeta_n (t)\}_{t \geq 0}$ is itself a stochastic process (on $\mathbb{R}$) driven by the stochastic differential equation
\[
\dd \zeta_n (t) = - \kappa \zeta_n (t) \dd t + \beta \dd m_n(t) \,,
\] 
with $\beta, \kappa > 0$ and 
\[
m_n(t) = \frac{1}{n} \sum_{i=1}^n \sigma_i(t) \,.
\]
From a formal viewpoint, the pair $\{(\sigma(t), \zeta_n (t))\}_{t \geq 0}$ is a Markov process  on \mbox{$\{-1,+1\}^n \times \mathbb{R}$}, with infinitesimal generator

\begin{equation}\label{CWdiss:micro:gen}
\cG_n f (\sigma, \xi) = \sum_{i=1}^n \Gamma(-\sigma_i \xi) \left[ f \left( \sigma^i, \xi - \frac{2\beta\sigma_i}{n} \right) - f(\sigma,\xi)\right] - \kappa \xi \partial_{\xi} f(\sigma,\xi) \,.
\end{equation}

The expression \eqref{CWdiss:micro:gen} describes a system of mean-field ferromagnetically coupled spins, in which the interaction energy is dissipated over time. The parameter $\beta$ represents the inverse temperature, while $\kappa$ tunes the intensity of dissipation. Notice that the choice $\Gamma(z) = 1+\tanh (z)$ gives a simplified version of the model introduced in \cite{dfr}. Moreover, by setting $\kappa =0$ and $\Gamma(z)=\exp(z)$ we obtain a Glauber dynamics for the classical Curie-Weiss model. \\ 

Let $E_n$ be the image of $\{-1,+1\}^n \times \mathbb{R}$ under the map $(\sigma, \zeta_n) \mapsto (m_n,\zeta_n)$. The evolution \eqref{CWdiss:micro:gen}, at the configuration level, induces Markovian dynamics on $E_n$ for the process $\{(m_n(t),\zeta_n(t))\}_{t \geq 0}$, that in turn evolves with generator

\begin{multline}\label{CWdiss:micro:gen:m}
\cA_n f(x,\xi) = \frac{n(1+x)}{2} \, \Gamma(-\xi) \left[ f \left( x - \frac{2}{n}, \xi - \frac{2\beta}{n} \right) - f(x,\xi) \right] \\
+ \frac{n(1-x)}{2} \, \Gamma(\xi) \left[ f \left( x + \frac{2}{n}, \xi + \frac{2\beta}{n} \right) - f(x,\xi) \right] - \kappa \xi \partial_{\xi} f(x,\xi)\,.
\end{multline}

\subsection{Law of large numbers and moderate deviations}

It is worth to mention that the methods in \cite{Kr16b,CoKr17,CoKr18} are not sufficient to obtain a path-space large deviation principle for the process $\{m_n(t),\zeta_n(t)\}_{t \geq 0}$ by the Feng-Kurtz approach \cite{FK06}. Indeed, the Hamiltonian is not of the standard type dealt with in \cite{CoKr17}, but we believe the comparison principle can be treated extending the methods in \cite{DFL11,KrReVe18}.
However, we can derive the infinite volume dynamics for our model via weak convergence.

\begin{theorem}[Law of large numbers]
\label{thm:LLN}
Suppose that $(m_n(0),\zeta_n(0))$ converges weakly to the constant $(m_0,\zeta_0)$. Then the process $\{m_n(t),\zeta_n(t)\}_{t \geq 0}$ converges weakly in law on $D_{\mathbb{R}^2}(\mathbb{R}^+)$ to the unique solution of
\begin{equation}\label{CWdiss:macro:dyn}
\left\{
\begin{array}{l}
\dot{m}(t) =   \Gamma (\zeta (t)) -  \Gamma (- \zeta (t)) - m(t) \big[ \Gamma (\zeta(t)) +  \Gamma (- \zeta (t)) \big] \\[0.1cm]
\dot{\zeta}(t) = \beta \Big\{ \Gamma (\zeta (t)) -  \Gamma (- \zeta (t))  - m(t) \big[ \Gamma (\zeta(t)) +  \Gamma (- \zeta (t)) \big] \Big\}  - \kappa \zeta (t),
\end{array}
\right.
\end{equation}
with initial conditions $(m(0),\zeta(0))=(m_0,\zeta_0)$. 
\end{theorem}

Under the assumption
\begin{assumption}\label{assumption_Gamma}
The function $\Gamma: \mathbb{R} \to \mathbb{R}$ is positive, increasing and such that (i) $\Gamma \in C^6(\mathbb{R})$; (ii) $\Gamma (0) \neq 0$, $\Gamma' (0) \neq 0$ and $\Gamma''(0) \geq 0$; (iii) the mapping $u \mapsto\frac{\Gamma(u) - \Gamma(-u)}{\Gamma(u) + \Gamma(-u) + \kappa}$ has at most one inflection point.
\end{assumption} 
we can characterize the phase space $(\kappa,\beta)$ for the dynamical system \eqref{CWdiss:macro:dyn}. It is described in the next theorem, whose proof is postponed to Appendix~\ref{appendix:proof_thm_phase_diagram}. The two depicted scenarios are then qualitatively summarized in the phase diagrams presented in Figure~\ref{fig:phase_diagram}.

\begin{theorem}[Phase diagram]\label{thm:phase_diagram}
For every $\kappa > 0$, define the line $\beta_{\mathrm{c}}(\kappa) = \left( 2\Gamma'(0) \right)^{-1} \left( \kappa + 2 \Gamma(0) \right)$. We have the following:
\begin{enumerate}[(I)]
\item 
Suppose that $\Gamma'''(0) < 0$. Then,
\begin{enumerate}[(A)]
\item if $\beta \leq \beta_{\mathrm{c}}(\kappa)$ the origin is a global attractor for \eqref{CWdiss:macro:dyn};
\item if $\beta > \beta_{\mathrm{c}}(\kappa)$ the system \eqref{CWdiss:macro:dyn} has a unique periodic orbit, that attracts all trajectories except for the fixed point $(0,0)$.
\end{enumerate}
\item 
Suppose that $\Gamma'''(0) > 0$ and set $\kappa_{\mathrm{tc}}:= \frac{6\Gamma''(0) \Gamma'(0)}{\Gamma'''(0)} - 2 \Gamma(0)$. Then,
\begin{enumerate}[(A)]
\item if $\kappa < \kappa_{\mathrm{tc}}$ the same result as in (I) holds; 
\item if $\kappa > \kappa_{\mathrm{tc}}$ there exists a further curve $0 < \beta_{\star}(\kappa) \leq \beta_{\mathrm{c}}(\kappa)$, separating at the tri-critical point $(\kappa_{\mathrm{tc}}, \beta_{\mathrm{c}}(\kappa_{\mathrm{tc}}))$, such that
\begin{enumerate}[(1)]
\item for $0 < \beta < \beta_{\star}(\kappa)$ the origin is a global attractor for \eqref{CWdiss:macro:dyn};
\item for $\beta_{\star}(\kappa) \leq \beta < \beta_{\mathrm{c}}(\kappa)$ the origin is locally stable and coexists with a stable limit cycle; 
\item for $\beta \geq \beta_{\mathrm{c}}(\kappa)$ the system \eqref{CWdiss:macro:dyn} has a unique periodic orbit that attracts all the trajectories except for the fixed point $(0,0)$.
\end{enumerate}
\end{enumerate}
\end{enumerate}
\end{theorem}

We want to point out that even if the two scenarios in Theorem~\ref{thm:phase_diagram} look naively similar, in the sense that they both describe a transition ``fixed point to limit cycle'', this is not in fact the case. They are qualitatively very different! In case (I) (or analogously (IIA)) a small-amplitude periodic orbit bifurcates from the origin through a (supercritical) Hopf bifurcation; whereas, in the setting (IIB) the limit cycle arises through a \emph{global bifurcation}, rather than a local one. More precisely, we have a concomitance of a \emph{saddle-node bifurcation of periodic orbits} and a (subcritical) Hopf bifurcation. This situation is always much more dramatic:  solutions that used to remain close to the origin are now forced to grow into large-amplitude oscillations.

\begin{remark*}
Assumption~\ref{assumption_Gamma} is the minimal assumption to guarantee existence and uniqueness of a stable periodic orbit in the supercritical regime $\beta > \beta_{\mathrm{c}}(\kappa)$. It is clear that by playing with the features of the function $\Gamma$ we can produce a customized number of phase transitions and coexisting stable limit cycles \cite{ChLlZh07,Oda96}.  See \cite{AnTo18} for a related example in this spirit.
\end{remark*}

\begin{example}
If we set $\Gamma (z) = 1 + \tanh (z)$, the analysis of the phase diagram done in~\cite[Sect.~3]{dfr} is recovered. Observe, in particular, that the latter choice for $\Gamma$ provides an example of scenario (I) in Theorem~\ref{thm:phase_diagram}. We can pick $\Gamma(z) = \exp (z)$ to have an example of scenario (II) instead.
\end{example}

\begin{figure}[h!t]
\centering%
\subfigure[]{\includegraphics[height=4.4cm,width=0.45\textwidth]{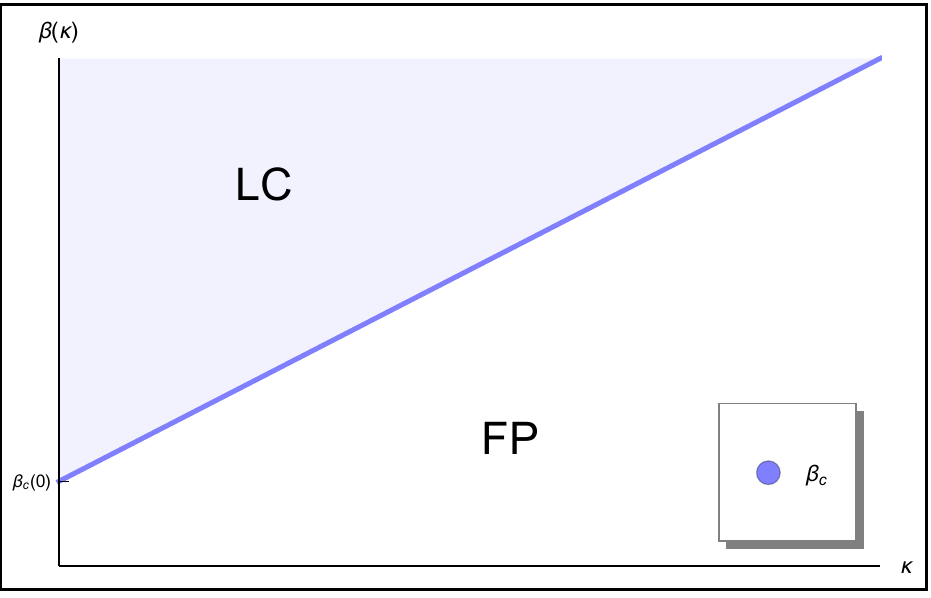}} 
\hspace{.3cm}
\subfigure[]{\includegraphics[height=4.4cm,width=0.45\textwidth]{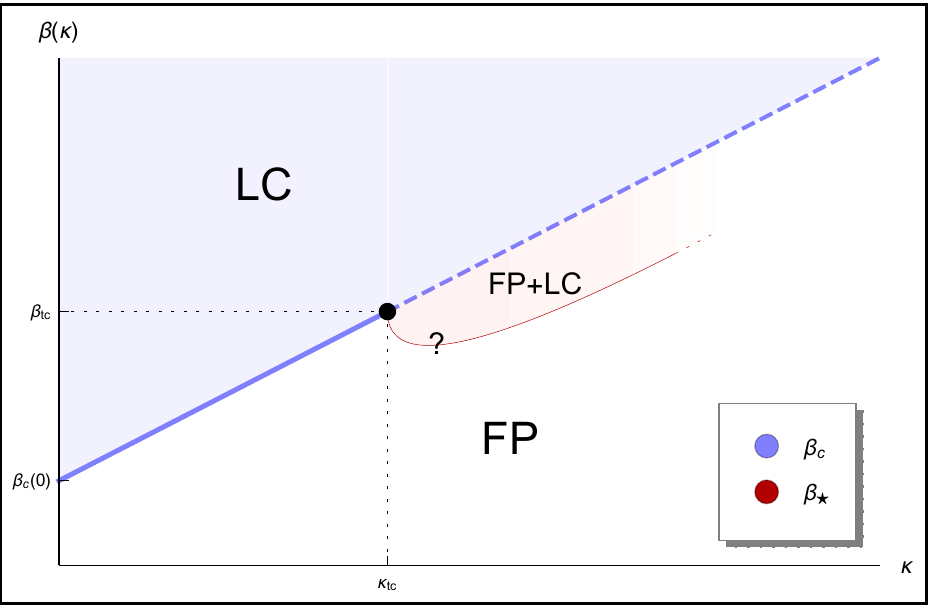}}
\caption{Phase space $(\kappa,\beta)$ for the dynamical system \eqref{CWdiss:macro:dyn}. Scenarios (I) (left panel) and (II) (right panel) of Theorem~\ref{thm:phase_diagram} are illustrated.  Each coloured region represents a phase with attractor(s) indicated by the alphabetical label: FP = fixed point; LC = limit cycle; FP+LC = coexistence of fixed point and limit cycle. The blue separation line corresponds to the Hopf bifurcation line; it is solid when the bifurcation is supercritical and dashed otherwise. The red curve is qualitative and represents the saddle-node bifurcation curve.}
\label{fig:phase_diagram}
\end{figure}

\begin{remark*}
With the same arguments as in Appendix~\ref{appendix:proof_thm_phase_diagram}, it is easy to prove that in the case $\Gamma''(0)<0$ (cf. Assumption~\ref{assumption_Gamma}), the sub- and supercritical Hopf bifurcations swap, giving rise to two ``reverse'' phase diagrams: the first where, for all values of $\kappa > 0$, by increasing $\beta$ we enconter first a saddle-node bifurcation of periodic orbits and then a sub-critical Hopf bifurcation and the second which is qualitatively the reflection about the line $\kappa=\kappa_{\mathrm{tc}}$ of scenario (II) in Theorem~\ref{thm:phase_diagram}.\\
We do not discuss these scenarios in details since we do not have reasonable examples of functions $\Gamma$ leading to such an output for the macroscopic dynamics.
\end{remark*}

If $\Gamma(z) = 1 + \tanh(z)$, central limit theorem and critical fluctuations have been studied in \cite{DaPTo18}. In the case of arbitrary function $\Gamma$, these latter results can be derived similarly. 

We want to study moderate deviations from the origin for the process $\{(m_n(t),\zeta(t))\}_{t \geq 0}$ in the various regimes. All the statements will be derived under Assumption~\ref{assumption_Gamma}.

The next result is mainly of interest in the phase(s) $\beta < \frac{\kappa + 2 \Gamma (0)}{2 \Gamma'(0)}$. 

\begin{theorem}[Moderate deviations around the origin]
\label{thm:subcrt:md:CWdiss}
Let $\{b_n\}_{n\geq 1}$ be a sequence of positive real numbers such that $b_n \to \infty$ and $b_n^{2} n^{-1} \to 0$. Suppose that $(b_n m_n(0), b_n \zeta_n(0))$ satisfies a large deviation principle with speed $nb_n^{-2}$ on $\mathbb{R}^2$ and rate function $I_0$. Then the trajectories $\left\{ \left( b_n m_n(t), b_n\zeta_n(t) \right) \right\}_{t \geq 0}$ satisfy the large deviation principle 
\[
\mathbb{P} \left[ \left\{ \left( b_n m_n(t), b_n \zeta_n(t) \right) \right\}_{t \geq 0} \approx \left\{ \gamma (t) \right\}_{t \geq 0} \right] \approx e^{-nb_n^{-2} I(\gamma)} 
\]
on $D_{\mathbb{R} \times \mathbb{R}}(\mathbb{R}^+)$, with good rate function
\[
I(\gamma) = 
\left\{
\begin{array}{ll}
I_0(\gamma(0)) + \int_0^{+\infty} \, \mathcal{L} (\gamma(s), \dot{\gamma}(s)) \, \dd s & \text{ if } \gamma \in \mathcal{AC},\\ 
\infty & \text{ otherwise},
\end{array}
\right.
\] 
where 
\[
\mathcal{L}((x,y),(v_x,v_y)) = 
\left\{
\begin{array}{ll}
\frac{1}{8 \Gamma(0)} \left\vert v_x - 2 \big( \Gamma'(0)y-\Gamma(0)x \big) \right\vert^2& \text{ if } v_y = \beta v_x - \kappa y\\
 \infty & \text{ otherwise}.
\end{array}
\right.
\]
\end{theorem}

We analyze more closely what is occurring at criticality. In the rest of the section we will always take the parameters $\kappa$ and $\beta$ so that $\kappa = 2 \beta \Gamma'(0) - 2 \Gamma(0)$. Moreover, for notational convenience, we will phrase all the results by swapping the role of $\kappa$ and $\beta$ (compared to the phase diagram description). For any sequence $\{ b_n \}_{n \geq 1}$ of positive real numbers such that $b_n \to \infty$, let us consider the process $\left\{ M_n(t), Z_n(t) \right\}_{t \geq 0}$ with 
\begin{equation}\label{rescaled_process:CWdiss} 
M_n(t) := \frac{b_n \left[ \beta m_n (t) - \zeta_n (t) \right]}{\sqrt{\Gamma(0) \big( \beta\Gamma'(0) - \Gamma(0) \big)}} 
\quad \mbox{ and } \quad
Z_n(t) := \frac{b_n \zeta_n (t)}{\Gamma(0)}\,.
\end{equation}
%
%
The intuitive idea behind this change of variables is the following. The dominant behavior of the pair $(m_n(t),\zeta_n(t))$ is driven by the linearization of \eqref{CWdiss:macro:dyn} around $(0,0)$. The origin is a center for the linearized system and the orbits are elliptic. We want to transform ellipses into circles to use polar coordinates; this transformation is equivalent to the change of variables described in \eqref{rescaled_process:CWdiss}. Thus, we  describe the behaviour of the pair $\left( M_n(t), Z_n(t) \right)$ through the polar coordinates $\left( R_n(t), \Theta_n(t) \right)$, where
\[
M_n(t):= \sqrt{R_n(t)} \, \sin \Theta_n(t) \quad \mbox{ and } \quad Z_n(t):=  \sqrt{R_n(t)} \, \cos \Theta_n(t).
\]
We want to understand if moderate deviations reflect the separation of time scales for the evolutions of $\{R_n(t)\}_{t \geq 0}$ and $\{\Theta_n(t)\}_{t \geq 0}$ that is highlighted for critical fluctuations in \cite{DaPTo18}. \\

\begin{theorem}[Moderate deviations: critical line $\kappa = 2 \beta \Gamma'(0) - 2 \Gamma(0)$]
\label{thm:crt:md:CWdiss} 
Assume \mbox{$3 \Gamma''(0) - \beta \Gamma''' (0) \geq 0$}. Let $\{b_n\}_{n\geq 1}$ be a sequence of positive real numbers such that $b_n \to \infty$ and $b_n^{4} n^{-1} \to 0$.  Suppose that $b_n R_n(0)$ satisfies the large deviation principle with speed $nb_n^{-4}$ and rate function $I_0 : \bR^+ \rightarrow [0,\infty]$. Then the trajectories $\left\{ R_n(b_n^2 t) \right\}_{t \geq 0}$ satisfy the large deviation principle 
\[
\mathbb{P} \left[ \left\{ R_n(b_n^2 t) \right\}_{t \geq 0} \approx \left\{ \gamma (t) \right\}_{t \geq 0} \right] \approx e^{-n^{1-4\alpha} I(\gamma)} 
\]
on $D_{\mathbb{R}^+}(\mathbb{R}^+)$, with good rate function
\[
I(\gamma) = 
\left\{
\begin{array}{ll}
I_0(\gamma(0)) + \int_0^{+\infty} \, \mathcal{L} (\gamma(s), \dot{\gamma}(s)) \, \dd s & \text{ if } \gamma \in \mathcal{AC},\\ 
\infty & \text{ otherwise},
\end{array}
\right.
\] 
where 
\begin{equation}\label{eqn:definition_Lagrangian:critical_line}
\mathcal{L}(x,v) = \begin{cases}
\frac{\Gamma(0)}{16 \beta^2 x} \left[v + \frac{\Gamma^2(0)}{4} \, \big( 3 \Gamma''(0) - \beta \Gamma''' (0) \big) x^2 \right]^2  & \text{if } x \neq 0, \\
0 & \text{if } x = 0, v = 0, \\
\infty & \text{if } x = 0, v \neq 0.
\end{cases}
\end{equation}
\end{theorem}

\begin{remark*}
The point $(0,0)$ is not an absorbing state for the microscopic evolution \eqref{CWdiss:micro:gen:m}.  We stress that the constraint $\mathcal{L}(0,v)= \infty$, for $v \neq 0$, does not imply that the trajectories cannot leave the origin with finite cost. For instance, the trajectory $t \mapsto t^2$, with $t \in [0,1]$, has finite cost.\\
\end{remark*}	

\begin{example}
Let $x \neq 0$. If we take $\Gamma(z) = 1+\tanh(z)$, we obtain $\mathcal{L}(x,v) = \frac{1}{16 \beta^2 x} \left[ v + \frac{\beta}{2} x^2 \right]^2$. Whereas, if we set $\Gamma(z) = \exp(z)$, we have $\mathcal{L}(x,v) = \frac{1}{16 \beta^2 x} \left[ v + \frac{1}{4} (3-\beta) x^2 \right]^2$. 
\end{example}

Note that the Lagrangian \eqref{eqn:definition_Lagrangian:critical_line} trivializes if $\beta =  \frac{3 \Gamma''(0)}{\Gamma'''(0)}$. For instance, this can happen at $\beta=3$ if $\Gamma (z) = \exp(z)$; on the contrary, it cannot happen if $\Gamma(z)=1+\tanh(z)$. To further study the fluctuations of $\{R_n(t)\}_{t \geq 0}$, we pose ourselves at the tri-critical point $(\kappa_{\mathrm{tc}},\beta_{\mathrm{tc}})$, with $\kappa_{\mathrm{tc}} = \frac{6\Gamma''(0)\Gamma'(0)}{\Gamma'''(0)} -2\Gamma(0)$ and $\beta_{\mathrm{tc}} = \frac{\kappa_{\mathrm{tc}}+2\Gamma(0)}{2\Gamma'(0)} = \frac{3\Gamma''(0)}{\Gamma'''(0)}$, and we speed up time to capture higher order effects of the microscopic dynamics.

\begin{theorem}[Moderate deviations: tri-critical point $\kappa = \frac{6\Gamma''(0)\Gamma'(0)}{\Gamma'''(0)} -2\Gamma(0)$ and $\beta = \frac{3\Gamma''(0)}{\Gamma'''(0)}$]
\label{thm:tri-crt:md:CWdiss}
Assume $5 \Gamma^{(4)}(0) \Gamma'''(0) - 3 \Gamma^{(5)}(0) \Gamma''(0) \geq 0$. Let $\{b_n\}_{n\geq 1}$ be a sequence of positive real numbers such that $b_n \to \infty$ and $b_n^{6} n^{-1} \to 0$.  Suppose that $b_n R_n(0)$ satisfies the large deviation principle with speed $nb_n^{-6}$ and rate function $I_0 : \bR^+ \rightarrow [0,\infty]$. Then the trajectories $\left\{ R_n(b_n^4 t) \right\}_{t \geq 0}$ satisfy the large deviation principle 
\[
\mathbb{P} \left[ \left\{ R_n(b_n^4t) \right\}_{t \geq 0} \approx \left\{ \gamma (t) \right\}_{t \geq 0} \right] \approx e^{-n^{1-6\alpha} I(\gamma)} 
\]
on $D_{\mathbb{R}^+}(\mathbb{R}^+)$, with good rate function
\[
I(\gamma) = 
\left\{
\begin{array}{ll}
I_0(\gamma(0)) + \int_0^{+\infty} \, \mathcal{L} (\gamma(s), \dot{\gamma}(s)) \, \dd s & \text{ if } \gamma \in \mathcal{AC},\\ 
\infty & \text{ otherwise},
\end{array}
\right.
\] 
where 
\[
\mathcal{L}(x,v) =  
\begin{cases}
\frac{\Gamma(0) \left( \Gamma'''(0)\right)^2}{144 \left( \Gamma''(0) \right)^2 x} \left\vert v + \frac{\Gamma^{(4)}(0)}{96 \Gamma'''(0)} \left(5 \Gamma^{(4)}(0) \Gamma'''(0) - 3 \Gamma^{(5)}(0) \Gamma''(0) \right) x^3 \right\vert^2 & \text{if } x \neq 0,\\
0 & \text{if } x = 0, v = 0, \\
\infty & \text{if } x = 0, v \neq 0.
\end{cases}
\]

\end{theorem}

The sign of the quantity $5 \Gamma^{(4)}(0) \Gamma'''(0) - 3 \Gamma^{(5)}(0) \Gamma''(0) $ determines if at the tri-critical point the Hopf bifurcation is sub- or supercritical. In particular, in the case when $5 \Gamma^{(4)}(0) \Gamma'''(0) - 3 \Gamma^{(5)}(0) \Gamma''(0) > 0$, we have supercriticality. If we are considering a function $\Gamma$ for which such a quantity vanishes, we can go to the next order and repeat the same reasoning.

\section{Proofs}\label{sect:proofs}

\subsection{Proof of Theorem~\ref{thm:LLN}}

The proof is based on a combination of proving a compact containment condition and the convergence of generators. 
	
\paragraph{Compact containment condition.} Observe that the process $\{m_n(t)\}_{t \geq 0}$ is confined in the compact $[-1,+1]$ for all $n \in \mathbb{N}$. Thus, we have to show compact containment for the process $\{\beta m_n(t) - \zeta_n(t) \}_{t \geq 0}$ only.\\
For every $n \in \mathbb{N}$, we define the stopping time $\tau_n^M := \inf \left\{ t \geq 0 : \left\vert \beta m_n(t) - \zeta_n(t) \right\vert \geq M \right\}$. We study the asymptotic behavior of the sequence $\left\{ \tau_n^M \right\}_{n \geq 1}$. 

\begin{lemma}\label{lmm:control_on_stopping_times}
	For any $T \geq 0$ and $\varepsilon > 0$, there exists 
	a strictly positive constant $M_{\varepsilon}$ such that $\sup_{n \geq 1} \mathbb{P}\left( \tau_{n}^{M_{\varepsilon}} \leq T \right) \leq \varepsilon$.
\end{lemma}

\begin{proof}
	Let $M$ be an arbitrary strictly positive constant. Notice that we have
	\begin{equation}\label{eqn:event_bounded_stopping_time}
	\mathbb{P} \left( \tau_n^M \leq T \right) = \mathbb{P} \left( \sup_{0 \leq t \leq T \wedge \tau_n^M} \left\vert \beta m_n(t) - \zeta_n(t) \right\vert \geq M \right).
	\end{equation}
	We will obtain a bound for \eqref{eqn:event_bounded_stopping_time} and we will show that it can be made arbitrarily small. 
	
	Consider the function $\Upsilon(x,\xi) := \frac{1}{2} (\beta x - \xi)^2$. Since $\Upsilon \big(x \pm \frac{2}{n}, \xi \pm \frac{2\beta}{n} \big) = \Upsilon(x,\xi)$, the evolution of the process $\Upsilon(m_n(t),\zeta_n(t))$ turns out to be deterministic, driven by the infinitesimal generator   
	\begin{equation*}
	\mathcal{A}_n \Upsilon (x,\xi) = - \kappa \xi \, \partial_\xi \Upsilon(x,\xi) = - \kappa \xi^2 + \beta \kappa x \xi \, ,
	\end{equation*}
cf. equation \eqref{CWdiss:micro:gen:m}.	Note that $\mathcal{A}_n \Upsilon(x,\xi) \leq \frac{\kappa \beta^2}{4}$ (by using $\sup_{z \in \mathbb{R}} -a z^2 + bz = \frac{b^2}{4a}$, whenever $a >0$ and $b \in \mathbb{R}$). As a consequence, we get 
	\begin{equation*}
	\Upsilon(m_n(t),\zeta_n(t)) = \Upsilon(m_n(0),\zeta_n(0)) + \int_0^t \mathcal{A}_n \Upsilon(m_n(s), \zeta_n(s)) ds \leq  \Upsilon(m_n(0),\zeta_n(0)) + t \, \frac{\kappa \beta^2 }{4},
	\end{equation*}
	leading to
\begin{align*}
\mathbb{P} \left( \sup_{0 \leq t \leq T \wedge \tau_n^M} \left\vert  \beta m_n(t) - \zeta_n(t) \right\vert \geq M \right) &= \mathbb{P} \left( \sup_{0 \leq t \leq T \wedge \tau_n^M} \Upsilon(m_n(t),\zeta_n(t)) \geq \frac{M^2}{2} \right) \\
&\leq \mathbb{P} \left( \Upsilon(m_n(0),\zeta_n(0)) \geq \frac{M^2}{2} - T \, \frac{\kappa \beta^2}{4}\right). 
	\end{align*}
	The convergence in law of the initial conditions implies $\mathbb{P} \left(  \Upsilon(m_n(0),\zeta_n(0)) \geq c_0(\varepsilon) \right) \leq \varepsilon$ for a sufficiently large $c_0(\varepsilon) > 0$ and for all $n \in \mathbb{N}$. Therefore, for any $\varepsilon > 0$, by choosing the constant $M = M_{\varepsilon}$ such that $\frac{M^2}{2} - T \, \frac{\kappa \beta^2}{4} \geq c_0(\varepsilon)$, we obtain $\sup_{n \geq n_{\varepsilon}} \mathbb{P} \left( \tau_n^{M_{\varepsilon}} \leq T \right) \leq \varepsilon$ as wanted.
\end{proof}

\paragraph{Convergence of the sequence of generators.} The infinitesimal generator of the process $\{(m_n(t),\zeta_n(t))\}_{t \geq 0}$ is given by \eqref{CWdiss:micro:gen:m}. We want to characterize the limit of the sequence $\mathcal{A}_n f$ for $f \in  C_c^2([-1,1] \times \bR)$. We Taylor expand $f$ up to first order 
\[
f \left( x \pm \tfrac{2}{n}, \xi \pm \tfrac{2\beta}{n} \right) - f(x, \xi) = \pm \tfrac{2}{n} \, \partial_x f(x,\xi) \pm \tfrac{2\beta}{n} \, \partial_{\xi} f(x,\xi) + o(1) 
\]
and we get
\begin{multline}\label{CWdiss:micro:gen:m:expanded}
\mathcal{A}_n f(x,\xi) = \left\{ \Gamma(\xi) - \Gamma(-\xi) - x \big( \Gamma(\xi) + \Gamma(-\xi) \big) \right\} \partial_x f(x,\xi) \\
+ \left\{ \beta \left[ \Gamma(\xi) - \Gamma(-\xi) - x \big( \Gamma(\xi) + \Gamma(- \xi) \big) \right] -\kappa \xi \right\} \partial_{\xi} f(x,\xi)+ o(1).
\end{multline}
Let $\mathcal{A}$ be the linear generator
\begin{multline}\label{CWdiss:micro:gen:m:limit}
\mathcal{A}f(x,\xi) = \left\{ \Gamma(\xi) - \Gamma(-\xi) - x \big( \Gamma(\xi) + \Gamma(-  \xi) \big) \right\} \partial_x f(x,\xi) \\
+ \left\{ \beta \left[ \Gamma(\xi) - \Gamma(-\xi) - x \big( \Gamma(\xi) + \Gamma(- \xi) \big) \right] -\kappa \xi \right\} \partial_{\xi} f(x,\xi).
\end{multline}
Since, for every $f \in  C_c^2([-1,1] \times \bR)$ and any compact set $K \subset \mathbb{R}^2$, we have 
\[
\lim_{n \to \infty} \sup_{(x,\xi) \in K \cap E_n} \left\vert \mathcal{A}_n(x,\xi) - \mathcal{A}(x,\xi) \right\vert = 0,
\]
we obtain the convergence of $\mathcal{A}_n$ to $\mathcal{A}$, as $n$ tends to infinity.\\
 
To conclude and prove the weak convergence result, we verify the conditions needed to apply \cite[Cor.~4.8.16]{EtKu86}. By Lemma~\ref{lmm:control_on_stopping_times} the sequence $\{(m_n,\zeta_n)\}_{n \geq 1}$ satisfies the compact containment condition. Additionally, the set $C_c^2([-1,1] \times \bR)$ is an algebra that separates points. We are left to show that uniqueness to the martingale problem for $(\mathcal{A},C_c^2([-1,1] \times \bR))$ holds. Lacking a good reference for this last result, we give a list of results from which this can be derived. 

First of all, uniqueness of the martingale problem follows from the comparison principle for the Hamilton-Jacobi equation $f - \lambda \cA f = h$, with $h \in C_b([-1,1]\times \bR)$ and $\lambda > 0$, by \cite[Thm.~3.7]{CoKu15}. 
In turn, we get the comparison principle  by a version of \cite[Prop.~3.5]{CoKr17}, which holds for operators that admit a good containment function and are of the form $\ip{\nabla f}{\bF}$, with $\mathbb{F}$ locally Lipschitz vector field. Indeed, the function $\Upsilon$ in the proof of Lemma \ref{lmm:control_on_stopping_times} is a good containment function for $\cA$ and, as $\Gamma$ is continuously differentiable, the vector field in \eqref{CWdiss:micro:gen:m:limit} is locally Lipschitz. This concludes the proof of Theorem~\ref{thm:LLN}.\\

\bigskip

We now turn on proving the moderate large deviation principles given in Theorems~\ref{thm:subcrt:md:CWdiss}--\ref{thm:tri-crt:md:CWdiss}. Following the methods of \cite{FK06} we will study moderate deviations based on the convergence of Hamiltonians and well-posedness of a class of Hamilton-Jacobi equations corresponding to a limiting Hamiltonian. These techniques have been also applied in \cite{Kr16b,CoKr17,CoKr18,CoGoKr18,DFL11}.\\ 
For the result in Theorem~\ref{thm:subcrt:md:CWdiss}, considering moderate deviations for the pair $(m_n(t),\zeta_n(t))$, we will refer to the large deviation principle in \cite[App.~A]{CoKr17}. For the results in Theorems~\ref{thm:crt:md:CWdiss} and \ref{thm:tri-crt:md:CWdiss}, stated for a one dimensional process, we need a more sophisticated large deviation result, which is still based on the abstract framework introduced in \cite{FK06}. We recall the notions needed for the latter results in Appendix~\ref{appendix:large_deviations_for_projected_processes}.

\subsection{Proof of Theorem~\ref{thm:subcrt:md:CWdiss}}
\label{section:proof_simple_MDP}

The generator $\hat{\mathcal{A}}_n$ of the process $\left\{ \left( b_n m_n(t), b_n \zeta_n(t) \right) \right\}_{t \geq 0}$ is given by
\begin{multline*}
\hat{\mathcal{A}}_n f(x,\xi) = \frac{n(1+xb_n^{-1})}{2} \, \Gamma \left(-\xi b_n^{-1}\right) \left[ f \left( x - 2b_n n^{-1}, \xi - 2\beta b_n n^{-1} \right) - f(x,\xi) \right] \\
\hspace{1cm} + \frac{n(1-xb_n^{-1})}{2} \, \Gamma \left(\xi b_n^{-1}\right) \left[ f \left( x + 2b_n n^{-1}, \xi + 2\beta b_n n^{-1} \right) - f(x,\xi)\right] - \kappa \xi \partial_\xi f(x,\xi) \,.
\end{multline*}
Thus, at speed $nb_n^{-2}$, the Hamiltonian is $H_n f = b_n^2 n^{-1} e^{-nb_n^{-2}f} \hat{\mathcal{A}}_n e^{nb_n^{-2}f}$ and can be explicitly computed as
\begin{multline*}
H_n f(x,\xi) =  \frac{b_n^2(1+xb_n^{-1})}{2} \, \Gamma \left(-\xi b_n^{-1}\right) \left\{ e^{n b_n^{-2} \left[ f \left( x - 2b_n n^{-1}, \xi - 2\beta b_n n^{-1} \right) - f(x,\xi) \right]} -1 \right\} \\
+ \frac{b_n^2(1-xb_n^{-1})}{2} \, \Gamma \left(\xi b_n^{-1}\right) \left\{ e^{nb_n^{-2}\left[f \left( x + 2b_n n^{-1}, \xi + 2\beta b_n n^{-1} \right) - f(x,\xi)\right]} -1 \right\} - \kappa \xi \partial_\xi f(x,\xi) \,.
\end{multline*}
By combining Taylor expansions 
\[
\Gamma(\pm \xi b_n^{-1})= \Gamma(0) \pm \Gamma'(0) \, \xi b_n^{-1} + o\left(b_n^{-1} \right),
\]
which is uniform for $\xi$ on compact sets, and
\begin{multline*}
\exp \left\{ n b_n^{-2} \left[ f \left( x \pm 2b_n n^{-1}, \xi \pm 2\beta b_n n^{-1} \right) - f(x,\xi) \right] \right\} - 1\\
= \pm 2 b_n^{-1} \partial_x f(x,\xi) \pm 2 \beta b_n^{-1} \partial_\xi f(x,\xi) + 2 b_n^{-2} \left( \partial_x f(x,\xi) \right)^2 \\
+ 4 \beta b_n^{-2} \partial_x f(x,\xi) \partial_\xi f(x,\xi) + 2 \beta^2 b_n^{-2} \left( \partial_\xi f(x,\xi) \right)^2 + o \left( b_n^{-2}\right),
\end{multline*}
we get
\begin{multline*}
H_n f(x,\xi) = 2 \big[\Gamma'(0)\xi- \Gamma(0)x \big] \partial_x f(x,\xi) + \left\{ 2 \beta \big[\Gamma'(0)\xi- \Gamma(0)x \big] - \kappa \xi \right\} \partial_\xi f(x,\xi) \\
+ 2 \Gamma(0) \left\{ \left( \partial_x f(x,\xi) \right)^2 + 2 \beta \partial_x f(x,\xi) \partial_\xi f(x,\xi) + \beta^2 \left( \partial_\xi f(x,\xi) \right)^2 \right\}  + o \left( 1 \right).
\end{multline*}
Let $H$ be the non-linear generator 
\begin{multline*}
H f(x,\xi) = 2 \big[\Gamma'(0)\xi- \Gamma(0)x \big] \partial_x f(x,\xi) + \left\{ 2 \beta \big[\Gamma'(0)\xi- \Gamma(0)x \big] - \kappa \xi \right\} \partial_\xi f(x,\xi) \\
+ 2 \Gamma(0) \left\{ \left( \partial_x f(x,\xi) \right)^2 + 2 \beta \partial_x f(x,\xi) \partial_\xi f(x,\xi) + \beta^2 \left( \partial_\xi f(x,\xi) \right)^2 \right\}.
\end{multline*}
Since, for $f \in C^2_c(\mathbb{R}^2)$ and any compact set $K \subset \mathbb{R}^2$, we have 
\[
\lim_{n \to \infty} \, \sup_{(x,\xi) \in K \cap E_n} \left\vert H_n f(x,\xi) - H f(x,\xi) \right\vert = 0,
\]
we obtain the convergence of $H_n$ to $H$ as $n$ tends to infinity. Note that as $f \in C_c^2(\bR^2)$, we have a uniform bound on the error term $o(1)$ implying $\sup_n \vn{H_n f} < \infty$. The final result follows by applying \cite[Thm.~A.17, Lem.~3.4 and Prop.~3.5]{CoKr17}.

\subsection{Proof of Theorems \ref{thm:crt:md:CWdiss} and \ref{thm:tri-crt:md:CWdiss}}

To prove Theorems \ref{thm:crt:md:CWdiss} and \ref{thm:tri-crt:md:CWdiss} we rely on the abstract large deviation principle in Theorem~\ref{theorem:Abstract_LDP}, which is a projected large deviation principle. We will make use of the projection map 
\[
\begin{array}{cccl}
\eta_n: & \mathbb{R}^2 & \to & \mathbb{R}^+\\
            & (x,\xi) & \mapsto & (\pi_1 \circ \Phi)(x,\xi) = x^2 + \xi^2,
\end{array}
\]
where $\pi_1$ is the projection on the first coordinate and $\Phi$ the coordinate transformation defined in \eqref{definition:phi}, and of the results recalled in Appendix~\ref{appendix:large_deviations_for_projected_processes}. Through the mapping $\eta_n$ we first transform the pair $(M_n,Z_n)$ into a polar coordinate pair $(R_n,\Theta_n)$ and then project to characterize the component $R_n$ only. In particular, for Theorem~\ref{theorem:Abstract_LDP} to be applied we must check the following conditions:

\begin{enumerate}[(a)]
\item 
The processes $\{(R_n(b_n^{\nu}t),\Theta_n(b_n^{\nu}t))\}_{t \geq 0}$, with $\nu \in \{2,4\}$, satisfy an appropriate exponential compact containment condition. See Subsection~\ref{subsect:exponential_compact_containment}.
\item
There is an operator $H \subseteq C_c^{\infty}(\mathbb{R}^+) \times C_c^{\infty}(\mathbb{R}^+)$ such that $H \subseteq ex-\LIM H_n$. In other words, for all $(f,g)\in H$, we need to determine $f_n \in H_n$ such that $\LIM_n f_n = f$ and $LIM_n H_n f_n =g$. See Subsection~\ref{subsect:limiting_Hamiltonians}.\\
We refer the reader to Definition~\ref{def:definition_LIM} (resp. \ref{def:definition_exLIM}) for the notion of $\LIM$ (resp. $ex-\LIM$) and to grasp the role of the projection map $\eta_n$ in the convergence.
\item The comparison principle holds for the Hamilton-Jacobi equation $f -\lambda Hf = h$ for all $h \in C_b(\mathbb{R}^+)$ and all $\lambda>0$. See Section~\ref{section:comparison_principle_singular_hamiltonian}.
\end{enumerate}

A main ingredient to deduce the large deviation principles given in Theorems~\ref{thm:crt:md:CWdiss} and \ref{thm:tri-crt:md:CWdiss} is the convergence of the Hamiltonians corresponding to the process $\left\{ (R_n(t), \Theta_n(t)) \right\}_{t \geq 0}$. We start by deriving an expansion for the Hamiltonian associated to an arbitrary time rescaling of such a process and to an arbitrary speed. We will then use the expansion to obtain the results stated in Theorems~\ref{thm:crt:md:CWdiss} and \ref{thm:tri-crt:md:CWdiss}.

\subsubsection{Preliminaries: expansion of the Hamiltonian}
In the next few sections we will be working on the critical curve; so, we set $\kappa = 2\beta \Gamma'(0) - 2\Gamma(0)$.  Moreover, not to clutter too much our formulas, we introduce the following notations 
%
%
\begin{equation*}
\mathbb{K}_j[\Gamma, \beta] := \Gamma^{2j}(0) \left[ \beta \Gamma^{(2j+1)}(0) - (2j+1) \Gamma^{(2j)}(0) \right], \qquad  (j = 0, 1, 2)
\end{equation*}
and 
\begin{equation}\label{def:functions_O}
O_{i,j}(\theta) := \cos^i \theta \sin^j \theta, \qquad (i,j \geq 0; \theta \in S^1).
\end{equation}
We consider the fluctuation process $\left\{ \left( M_n(b_n^{\nu} t), Z_n(b_n^{\nu} t) \right)\right\}_{t \geq 0}$. Observe that, being $(m_n,\zeta_n) \mapsto (M_n,Z_n)$ a linear and invertible transformation, the infinitesimal generator of the Markov pair $\left( M_n(b_n^{\nu}t), Z_n(b_n^{\nu}t)\right)$ defined in \eqref{rescaled_process:CWdiss} can be deduced from \eqref{CWdiss:micro:gen:m}. Indeed, taking into account the time speed up, we obtain
\begin{align}\label{generator_critical_curve}
A_n f (x,\xi) &= \frac{n b_n^{\nu}}{2} \left( 1 + \frac{\sqrt{\Gamma(0) \mathbb{K}_0[\Gamma,\beta]} \, x + \Gamma(0) \xi}{\beta b_n} \right) \, \Gamma \left( - \frac{\Gamma(0) \xi}{b_n} \right)  \left[ f\left( x, \xi - \frac{2\beta b_n}{\Gamma(0) n} \right) - f(x,\xi)\right]\nonumber \\
&+ \frac{n b_n^{\nu}}{2} \left( 1 - \frac{\sqrt{\Gamma(0) \mathbb{K}_0[\Gamma,\beta]} \, x + \Gamma(0) \xi}{\beta b_n}  \right) \, \Gamma \left( \frac{\Gamma(0) \xi}{b_n} \right) \left[ f\left( x, \xi + \frac{2\beta b_n}{\Gamma(0) n} \right) - f(x,\xi)\right] \nonumber\\
&+ 2 b_n^{\nu} \sqrt{\Gamma(0) \mathbb{K}_0[\Gamma,\beta]} \, \xi \partial_x f(x,\xi) - 2 b_n^{\nu} \, \mathbb{K}_0[\Gamma,\beta] \, \xi \partial_\xi f(x,\xi). 
\end{align}
We make a change of variables and we turn into polar coordinates. We define the function $\atan2: \mathbb{R}^2 \setminus \{(0,0)\} \rightarrow S^1$ as 
\begin{equation}\label{definition:atan2}
\atan2(x,\xi) :=
\begin{cases}
\arctan \frac{x}{\xi} & \text{ if } x > 0 \text{ and } \xi \geq 0, \\[0.2cm]
\arctan \frac{x}{\xi} + 2\pi & \text{ if } x > 0 \text{ and } \xi < 0, \\[0.2cm]
\arctan \frac{x}{\xi} + \pi & \text{ if } x < 0, \\[0.2cm]
\frac{\pi}{2} & \text{ if } x = 0 \text{ and } \xi > 0, \\[0.2cm]
\frac{3\pi}{2} & \text{ if } x = 0 \text{ and } \xi < 0 
\end{cases}
\end{equation}
and consider the coordinate transformation $\Phi : \bR^2 \rightarrow \{0\} \cup \left((0,\infty) \times S^1\right)$ defined by
	\begin{equation}\label{definition:phi}
	\Phi(x,\xi) = \begin{cases}
	0 & \text{if } x^2 + \xi^2 = 0, \\[0.1cm]
	(x^2 + \xi^2, \atan2(x,\xi)) & \text{if } x^2 + \xi^2 > 0.
	\end{cases}
	\end{equation}
	
The state-space after the transformation is $\fS := \{0\} \cup \left((0,\infty) \times S^1 \right)$ and comes equipped with the topology induced by the map \eqref{definition:phi}. In other words, $s_n \in \fS$ converges to $s \in \fS$ if and only if $\Phi^{-1}(s_n) \rightarrow \Phi^{-1}(s)$ in $\bR^2$. \\
Generally, we will write $(r,\theta)$ to denote an element in $\fS$. If $r = 0$, we will understand that we mean the element $\{0\}$, regardless of the angle $\theta$. To make sure this does not yield any issues, we will work with a class of functions that are constant on a neighbourhood of $(0,0)$. In this way, when expressing $H_n f$ in terms of $(r,\theta)$, the value $H_n f(0,\theta)$ is uniquely identified by $0$. 

Note also that our choice of polar coordinates is non-conventional. Instead of considering the usual radius and angle variables, we are using the pair radius squared and angle. This is for notational convenience and to be consistent with notation in \cite{DaPTo18}. Moreover, for any fixed radius, the parametric curve obtained in $\mathbb{R}^2$ is traversed clockwise, rather than anticlockwise, while increasing the angle in $S^1$.\\

We proceed now with the identification of the sequence of Hamiltonians associated with the polar coordinate process $\left\{ \left( R_n(b_n^{\nu} t), \Theta_n(b_n^{\nu} t) \right) \right\}_{t \geq 0}$,  where 
\[
R_n(b_n^{\nu} t) := M_n^2(b_n^{\nu} t) + Z_n^2(b_n^{\nu} t) \quad \mbox{ and } \quad \Theta_n(b_n^{\nu} t) := \atan2 \left( M_n(b_n^{\nu} t), Z_n(b_n^{\nu} t) \right).
\]
At speed $nb_n^{-\delta}$ the Hamiltonian is given by
\begin{equation}\label{def:definition_generic_Hamiltonian}
H_n f (r,\theta) = H_n (f \circ \Phi)(x,\xi) =  b_n^{\delta} n^{-1} \, e^{-n b_n^{-\delta} (f \circ \Phi)(x,\xi)} A_n e^{n b_n^{-\delta} (f \circ \Phi)(x,\xi)}.
\end{equation}
We can compute
%
\begin{align}\label{eqn:generic_Hamiltonian:first_expansion}
H_n (f \circ \Phi)(x,\xi)  &= \frac{b_n^{\nu+\delta}}{2} \left( 1 + \frac{\sqrt{\Gamma(0)\mathbb{K}_0[\Gamma,\beta]} \, x + \Gamma(0) \xi}{\beta b_n}  \right) \, \Gamma \left( - \frac{\Gamma(0) \xi}{b_n} \right) \times \nonumber\\[.3cm]
&\qquad \qquad \times  \left[ \exp \left\{\frac{n}{b_n^{\delta}} \left[ (f \circ \Phi)\left( x, \xi - \frac{2\beta b_n}{\Gamma(0)n} \right) - (f \circ \Phi)(x,\xi) \right] \right\} - 1 \right] \nonumber\\[.3cm]
&+ \frac{b_n^{\nu+\delta}}{2} \left( 1 - \frac{\sqrt{\Gamma(0)\mathbb{K}_0[\Gamma,\beta]} \, x + \Gamma(0) \xi}{\beta b_n} \right) \, \Gamma \left( \frac{\Gamma(0) \xi}{b_n} \right) \times \nonumber\\[.3cm]
& \qquad \qquad \times \left[ \exp \left\{\frac{n}{b_n^{\delta}} \left[ (f \circ \Phi)\left( x, \xi + \frac{2\beta b_n}{\Gamma(0)n} \right) - (f \circ \Phi)(x,\xi) \right] \right\} - 1 \right] \nonumber\\[.3cm]
&+ 2 b_n^{\nu} \sqrt{\Gamma(0)\mathbb{K}_0[\Gamma,\beta]} \, \xi \, \partial_x (f \circ \Phi)(x,\xi) - 2 b_n^{\nu} \, \mathbb{K}_0[\Gamma,\beta] \, \xi \, \partial_\xi (f \circ \Phi)(x,\xi). 
\end{align}
%
By combining Taylor expansions
\begin{multline*}
\Gamma \big( \pm \Gamma(0) \xi b_n^{-1} \big) = \Gamma(0) \pm \Gamma(0) \Gamma'(0) \xi b_n^{-1} + \tfrac{1}{2} \Gamma^2(0) \Gamma''(0) \xi^2 b_n^{-2} \pm \tfrac{1}{6} \Gamma^3(0) \Gamma'''(0) \xi^3 b_n^{-3} \\
+ \tfrac{1}{24} \Gamma^4(0) \Gamma^{(4)}(0) \xi^4 b_n^{-4} \pm \tfrac{1}{120} \Gamma^5(0) \Gamma^{(5)}(0) \xi^5 b_n^{-5} + o(b_n^{-5})
\end{multline*}
and
\begin{multline*}
\exp \left\{\frac{n}{b_n^{\delta}} \left[ (f \circ \Phi)\left( x, \xi \pm \frac{2\beta b_n}{\Gamma(0)n} \right) - (f \circ \Phi)(x,\xi) \right] \right\} - 1 \\
= \pm \tfrac{2 \beta}{\Gamma(0)} \, b_n^{1-\delta} \partial_\xi (f \circ \Phi)(x,\xi) + \tfrac{2 \beta^2}{\Gamma^2(0)} \, b_n^{2-2\delta} \left( \partial_\xi (f \circ \Phi)(x,\xi) \right)^2 + o(b_n^{2-2\delta}),
\end{multline*}
we get 
\begin{align*}
H_n (f \circ \Phi)(x,\xi) &= \Big[ -2 b_n^{\nu} \sqrt{\Gamma(0)\mathbb{K}_0[\Gamma,\beta]} \, x - b_n^{\nu-2}\Gamma(0) \Gamma''(0) \sqrt{\Gamma(0)\mathbb{K}_0[\Gamma,\beta]} \, x \xi^2 +\tfrac{1}{3} b_n^{\nu-2} \mathbb{K}_1[\Gamma,\beta] \, \xi^3   \\
 &- \tfrac{1}{12} b_n^{\nu-4}  \Gamma^3(0) \Gamma^{(4)}(0) \sqrt{\Gamma(0)\mathbb{K}_0[\Gamma,\beta]} \, x \xi^4   +\tfrac{1}{60} b_n^{\nu-4} \, \mathbb{K}_2[\Gamma,\beta] \, \xi^5 \Big] \partial_\xi (f \circ \Phi)( x, \xi)    \\
&+ 2 b_n^{\nu} \sqrt{\Gamma(0)\mathbb{K}_0[\Gamma,\beta]} \, \xi \partial_x (f \circ \Phi)(x,\xi) + \tfrac{2 \beta^2}{\Gamma(0)} b_n^{\nu-\delta+2} \left( \partial_\xi (f \circ \Phi)( x, \xi) \right)^2 \\
&+ o \left(b_n^{\nu-\delta+2} \right) + o(b_n^{\nu-4}). 
\end{align*}
To have an interesting remainder in the limit, we need $\delta = \nu + 2$. This gives
\begin{align*}
H_n (f \circ \Phi)(x,\xi) &= \left[ -2 b_n^{\nu} \sqrt{\Gamma(0)\mathbb{K}_0[\Gamma,\beta]} \, x - b_n^{\nu-2}\Gamma(0) \Gamma''(0) \sqrt{\Gamma(0)\mathbb{K}_0[\Gamma,\beta]} \, x \xi^2 +\tfrac{1}{3} b_n^{\nu-2} \mathbb{K}_1[\Gamma,\beta] \, \xi^3 \right.   \\
& \left. - \tfrac{1}{12} b_n^{\nu-4} \Gamma^3(0) \Gamma^{(4)}(0) \sqrt{\Gamma(0)\mathbb{K}_0[\Gamma,\beta]} \, x \xi^4   +\tfrac{1}{60} b_n^{\nu-4}  \, \mathbb{K}_2[\Gamma,\beta] \, \xi^5 \right] \partial_\xi (f \circ \Phi)(x, \xi)    \\
&+ 2 b_n^{\nu} \sqrt{\Gamma(0)\mathbb{K}_0[\Gamma,\beta]} \, \xi \partial_x (f \circ \Phi)(x,\xi) + \tfrac{2\beta^2}{\Gamma(0)} \left( \partial_\xi (f \circ \Phi)(x, \xi) \right)^2 + o(1) + o(b_n^{\nu-4})
\end{align*}
and the remainder term $o(1)+o(b_n^{\nu-4})$ converges uniformly to zero on sets $K \cap E_n$, with $K \subset \mathbb{R}^2$ compact. Now observe that
\[
\partial_x (f \circ \Phi)(x,\xi) = 2 x \partial_r f(r,\theta) + \frac{\xi }{x^2+\xi^2} \, \partial_{\theta} f(r,\theta) 
\]
and
\[
\partial_\xi (f \circ \Phi)(x,\xi)  = 2 \xi \partial_r f(r,\theta) - \frac{x}{x^2+\xi^2} \, \partial_{\theta} f(r,\theta). 
\]
Rearranging the terms and rephrasing the whole expression in terms of the only polar coordinates (recall \eqref{def:functions_O} and that the inverse coordinate transform reads $x = \sqrt{r} \sin(\theta), \xi = \sqrt{r} \cos(\theta)$), we obtain 
\begin{small}
\begin{align}\label{eqn:generic_Hamiltonian_expansion}
H_n f(r,\theta) &= \left[-2 b_n^{\nu-2} \Gamma(0) \Gamma''(0) \sqrt{\Gamma(0)\mathbb{K}_0[\Gamma,\beta]}  \, O_{3,1}(\theta) r^2 + \tfrac{2}{3} b_n^{\nu-2} \, \mathbb{K}_1[\Gamma,\beta] \, O_{4,0}(\theta) r^2  \right. \nonumber \\
&\left. -\tfrac{1}{6} b_n^{\nu-4} \Gamma^4(0) \Gamma^{(4)}(0) \sqrt{\Gamma(0)\mathbb{K}_0[\Gamma,\beta]} \, O_{5,1}(\theta) r^3 +\tfrac{1}{30} b_n^{\nu-4} \, \mathbb{K}_2[\Gamma,\beta] \, O_{6,0}(\theta) r^3 \right] \partial_r f(r,\theta) \nonumber \\
&+ \left[ 2 b_n^{\nu} \sqrt{\Gamma(0)\mathbb{K}_0[\Gamma,\beta]} + b_n^{\nu-2}\Gamma(0) \Gamma''(0) \sqrt{\Gamma(0)\mathbb{K}_0[\Gamma,\beta]} \, O_{2,2}(\theta) r - \tfrac{1}{3} b_n^{\nu-2} \, \mathbb{K}_1[\Gamma,\beta] \, O_{3,1}(\theta) r \right. \nonumber \\
&\left. +\tfrac{1}{12} b_n^{\nu-4} \Gamma^3(0) \Gamma^{(4)}(0) \sqrt{\Gamma(0)\mathbb{K}_0[\Gamma,\beta]} \, O_{4,2}(\theta) r^2 - \tfrac{1}{60} b_n^{\nu-4} \, \mathbb{K}_2[\Gamma,\beta] \, O_{5,1}(\theta) r^2  \right] \partial_{\theta} f(r,\theta) \nonumber \\
&+\tfrac{8\beta^2}{\Gamma(0)}  O_{2,0}(\theta) r \left(\partial_r f(r,\theta)\right)^2 - \tfrac{8\beta^2}{\Gamma(0)} O_{1,1}(\theta) \, \partial_r f(r,\theta) \partial_{\theta} f(r,\theta) + \tfrac{2\beta^2}{\Gamma(0)r} O_{0,2}(\theta) \left( \partial_{\theta} f(r,\theta)\right)^2 \nonumber \\
&+ o(1) + o(b_n^{\nu-4})
\end{align}
\end{small}
and the remainder term $o(1)+o(b_n^{\nu-4})$ converges uniformly to zero on compact sets of $\mathfrak{S}$.

\subsubsection{Perturbative approach and limiting Hamiltonians}
\label{subsect:limiting_Hamiltonians}

In our proofs $\nu \in \{2,4\}$. Observe that, for these values of $\nu$, the expansion \eqref{eqn:generic_Hamiltonian_expansion} is diverging and, more precisely, it is the time-evolution of the angular variable responsible for this divergence. Indeed, the two components of $\left\{ (R_n(t), \Theta_n(t)) \right\}_{t \geq 0}$ live on two different time-scales and the asymptotic behavior of $\left\{ R_n(t) \right\}_{t \geq 0}$ can be determined after having averaged out the evolution of $\left\{ \Theta_n(t) \right\}_{t \geq 0}$ with respect to the uniform measure on $S^1$. The ``averaging'' is obtained  through a perturbative approach leading to a projected large deviation principle. In other words, our aim is to show that the sequence of Hamiltonians $\{H_n\}_{n \geq 1}$ admits a limiting operator $H$ and, additionally, the graph of this limit depends only on the radial variable.\\

The perturbative argument takes inspiration from the perturbation theory for Markov processes introduced in \cite{PaStVa77} and it was also used to study path-space moderate deviations for the Curie-Weiss model with random field \cite{CoKr18} and with the ``self-organized criticality'' property \cite{CoGoKr18}.\\

We will first give some heuristics about the perturbative method and then we will make it rigorous.

\paragraph{Heuristics.} In the expansion \eqref{eqn:generic_Hamiltonian_expansion} the leading term is of order $b_n^\nu$ and thus explodes as $n \to \infty$. We think of $b_n^{-2}$ as a perturbative parameter and we use a first/second order perturbation $F_{n,f}$ of $f$ to introduce some negligible (in the infinite volume limit) terms providing that the whole expansion does not diverge and, moreover, converges to an averaged limit. More precisely, we have the following. \\

{\em Case $\nu=2$.} Given an arbitrary function $\Lambda_f^{(0)}: \mathfrak{S} \to \mathbb{R}$, we define the (first order) perturbation of $f$ as
\[
F_{n,f}^{(\nu=2)}(r,\theta) := f(r) + b_n^{-2} \Lambda_f^{(0)}(r,\theta)
\]
and we choose $\Lambda_f^{(0)}(r,\theta)$ so that
\[
H_n F_{n,f}^{(\nu=2)}(r,\theta) = H^{(\nu=2)} f(r) + \mbox{ remainder},
\]
where $H^{(\nu=2)} f(r)$ is of order $1$ with respect to $b_n$ and the remainder contains smaller order terms. We assume that $\Lambda_f^{(0)}$ is at least of class $C^2$ and we compute $H_n F_{n,f}^{(\nu=2)}$. First, for any $\theta \in S^1$ and $(r,p) \in \mathbb{R}^+ \times \mathbb{R}$, set 
\begin{multline}\label{pre-averaging_Hamiltonian:nu2_case}
H^{(\nu=2)}_\theta (r,p) := \left[ -2 \Gamma(0) \Gamma''(0) \sqrt{\Gamma(0)\mathbb{K}_0[\Gamma,\beta]} O_{3,1}(\theta) r^2 + \tfrac{2}{3} \mathbb{K}_1[\Gamma,\beta] O_{4,0}(\theta) r^2\right]p \\
+ \tfrac{8\beta^2}{\Gamma(0)}O_{2,0}(\theta) rp^2.
\end{multline} 
Hence, using \eqref{eqn:generic_Hamiltonian_expansion} yields
\[
H_n F_{n,f}^{(\nu=2)} (r,\theta) = H_{\theta}^{(\nu=2)}(r,f'(r)) + 2 \sqrt{\Gamma(0)\mathbb{K}_0[\Gamma,\beta]} \, \partial_{\theta} \Lambda_f^{(0)} (r,\theta) + \mbox{ remainder}.
\]
To average out the variable $\theta$ and get the desired limiting operator, the function $\Lambda_f^{(0)}$ must necessarily verify, for every $(r,\theta) \in \mathfrak{S}$, the relation
\begin{equation}\label{condition_first_order_perturbation}
H_{\theta}^{(\nu=2)}(r,f'(r)) + 2 \sqrt{\Gamma(0)\mathbb{K}_0[\Gamma,\beta]} \, \partial_{\theta} \Lambda_f^{(0)} (r,\theta) = H^{(\nu=2)} f(r), 
\end{equation}
where $H^{(\nu=2)} f(r) := \frac{1}{2\pi} \int_0^{2\pi} H_{\theta}^{(\nu=2)}(r,f'(r)) d\theta$.
If we take
\begin{equation}\label{definition_Lambda_0}
\Lambda_f^{(0)} (r,\theta) = \frac{1}{2\sqrt{\Gamma(0)\mathbb{K}_0[\Gamma,\beta]}} \left[ \theta H^{(\nu=2)} f(r) - \int_0^{\theta} H_{\alpha}^{(\nu=2)} (r,f'(r)) d\alpha \right],
\end{equation}
then the condition \eqref{condition_first_order_perturbation} is satisfied and we obtain
\[
H_n F_{n,f}^{(\nu=2)} (r,\theta) =  \tfrac{1}{4} \, \mathbb{K}_1[\Gamma,\beta]  \, r^2 f'(r) + \tfrac{4\beta^2}{\Gamma(0)} \, r \left( f'(r) \right)^2 + \mbox{ remainder}.
\]
Provided we can control the remainder, for any function $f$ in a suitable regularity class, we {\em formally} get the following candidate limiting operator
\begin{equation}\label{limiting_Hamiltonian:nu=2_case}
H^{(\nu=2)} f(r) = - \tfrac{1}{4} \Gamma^2(0) \left( 3\Gamma''(0) - \beta \Gamma'''(0) \right) r^2 f'(r) + \tfrac{4\beta^2}{\Gamma(0)}r \left( f'(r) \right)^2.
\end{equation}

{\em Case $\nu=4$.} Given two arbitrary functions $\Lambda_f^{(1)}, \Lambda_f^{(2)}: \mathfrak{S} \to \mathbb{R}$, we define the (second order) perturbation of $f$ as
\[
F_{n,f}^{(\nu=4)}(r,\theta) := f(r) + b_n^{-2} \Lambda_f^{(1)}(r,\theta) + b_n^{-4} \Lambda_f^{(2)}(r,\theta)
\]
and we repeat the same strategy as before. We choose $\Lambda_f^{(1)}(r,\theta)$ and $\Lambda_f^{(2)}(r,\theta)$ so that
\[
H_n F_{n,f}^{(\nu=4)}(r,\theta) = H^{(\nu=4)} f(r) + \mbox{ remainder},
\]
where $H^{(\nu=4)} f(r)$ is of order $1$ with respect to $b_n$ and the remainder contains smaller order terms. We assume that $\Lambda_f^{(1)}(r,\theta)$ and $\Lambda_f^{(2)}(r,\theta)$ are at least of class $C^2$ and we compute $H_n F_{n,f}^{(\nu=4)}$. Recall that, being at the tri-critical point, we have $\mathbb{K}_1[\Gamma,\beta] \equiv 0$ in this case. Moreover, for any $\theta \in S^1$ and $(r,p) \in \mathbb{R}^+ \times \mathbb{R}$, set 
\begin{multline}\label{pre-averaging_Hamiltonian:nu4_case}
H^{(\nu=4)}_\theta (r,p) := \left[  - \tfrac{1}{6} \Gamma^3(0) \Gamma^{(4)}(0) \sqrt{\Gamma(0) \mathbb{K}_0[\Gamma,\beta]} \, O_{5,1}(\theta) r^3 + \tfrac{1}{30} \mathbb{K}_2[\Gamma,\beta] \, O_{6,0}(\theta) r^3 \right]p \\
+ \tfrac{8\beta^2}{\Gamma(0)}O_{2,0}(\theta) rp^2.
\end{multline} 
Hence, using \eqref{eqn:generic_Hamiltonian_expansion} yields
\begin{align*}
H_n F_{n,f}^{(\nu=4)} (r,\theta) &= H_{\theta}^{(\nu=4)}(r,f'(r)) +2 b_n^{2} \sqrt{\Gamma(0)\mathbb{K}_0[\Gamma,\beta]} \left[ -\Gamma(0) \Gamma''(0) O_{3,1}(\theta) r^2  f'(r) + \partial_{\theta} \Lambda_f^{(1)} (r,\theta) \right] \\
&+ \Gamma(0) \Gamma''(0) \sqrt{\Gamma(0)\mathbb{K}_0[\Gamma,\beta]} \, O_{2,2}(\theta) r \, \partial_{\theta} \Lambda_f^{(1)} (r,\theta) + 2 \sqrt{\Gamma(0)\mathbb{K}_0[\Gamma,\beta]} \, \partial_{\theta} \Lambda_f^{(2)} (r,\theta) \\
&- 2  \Gamma(0) \Gamma''(0) \sqrt{\Gamma(0)\mathbb{K}_0[\Gamma,\beta]} \, O_{3,1}(\theta) r^2 \, \partial_r \Lambda_f^{(1)} (r,\theta)   +  \mbox{ remainder}. 
\end{align*}
Now we proceed in two steps. First, we choose the function $\Lambda^{(1)}_f$ to eliminate the diverging terms of order $b_n^2$. Observe that
\begin{equation}\label{definition_Lambda_1}
\Lambda_f^{(1)} (r,\theta) = - \tfrac{1}{4} \Gamma(0) \Gamma''(0) O_{4,0}(\theta) r^2 f'(r) 
\end{equation}
achieves the purpose. Given this choice and the identity 
\[
O_{5,3}(\theta)+O_{7,1}(\theta) = O_{5,1}(\theta) \big[O_{0,2}(\theta)+O_{2,0}(\theta) \big] = O_{5,1}(\theta),
\]
we obtain 
%
%
\begin{align*}
H_n F_{n,f}^{(\nu=4)} (r,\theta) &= H_{\theta}^{(\nu=4)}(r,f'(r)) + \Gamma^2(0) \left( \Gamma''(0) \right)^2 \sqrt{\Gamma(0)\mathbb{K}_0[\Gamma,\beta]} \, \left[ O_{5,1}(\theta) r^3 f'(r) + \tfrac{1}{2} O_{7,1}(\theta) r^4 f''(r) \right]\\
&+ 2 \sqrt{\Gamma(0)\mathbb{K}_0[\Gamma,\beta]} \, \partial_{\theta} \Lambda^{(2)}_f(r,\theta)   +  \mbox{ remainder}
\end{align*}
and to average out the variable $\theta$ and get the desired limiting operator, the function $\Lambda_f^{(2)}$ must necessarily verify, for every $(r,\theta) \in \mathfrak{S}$, the relation
\begin{multline}\label{condition_second_order_perturbation}
H_{\theta}^{(\nu=4)}(r,f'(r)) + \Gamma^2(0) \left( \Gamma''(0) \right)^2 \sqrt{\Gamma(0)\mathbb{K}_0[\Gamma,\beta]}  \left[ O_{5,1}(\theta) r^3 f'(r) + \tfrac{1}{2} O_{7,1}(\theta) r^4 f''(r) \right] \\
+ 2 \sqrt{\Gamma(0)\mathbb{K}_0[\Gamma,\beta]} \, \partial_{\theta} \Lambda^{(2)}_f(r,\theta) = H^{(\nu=4)}f(r),
\end{multline}
where 
\[
H^{(\nu=4)} f(r) := \frac{1}{2\pi} \int_0^{2\pi} \left[ H_{\theta}^{(\nu=4)}(r,f'(r)) + \Gamma^2(0) \left( \Gamma''(0) \right)^2 \sqrt{\Gamma(0)\mathbb{K}_0[\Gamma,\beta]} \, O_{5,1}(\theta) r^3 f'(r) \right] d\theta.
\]
Note that the second term in the integrand function clearly integrates to zero. We write it anyway to point out the possibility of having non-vanishing contributions to the Hamiltonian coming from the perturbation (see \cite[Thms.~2.8--2.12]{CoKr18} for a few examples). If we take 
\begin{multline}\label{definition_Lambda_2}
\Lambda_f^{(2)} (r,\theta) = \frac{1}{2\sqrt{\Gamma(0)\mathbb{K}_0[\Gamma,\beta]}} \left[ \theta H^{(\nu=4)} f(r) - \int_0^{\theta} H_{\alpha}^{(\nu=4)} (r,f'(r)) d\alpha \right.\\
+ \tfrac{1}{6} \Gamma^2(0) \left( \Gamma''(0)\right)^2 \sqrt{\Gamma(0)\mathbb{K}_0[\Gamma,\beta]} \, \left( O_{6,0}(\theta) - 1 \right) r^3 f'(r)\\
\left. + \tfrac{1}{16} \Gamma^2(0) \left( \Gamma''(0)\right)^2 \sqrt{\Gamma(0)\mathbb{K}_0[\Gamma,\beta]} \, \left( O_{8,0}(\theta) - 1 \right) r^4 f''(r) \right],
\end{multline}
then equation \eqref{condition_first_order_perturbation} is satisfied and we obtain
\[
H_n F_{n,f}^{(\nu=4)} (r,\theta) =  \tfrac{1}{96} \mathbb{K}_2[\Gamma,\beta]  \, r^3 f'(r) + \tfrac{4\beta^2}{\Gamma(0)} \, r \left( f'(r) \right)^2 + \mbox{ remainder}.
\]
Provided we can control the remainder, for any function $f$ in a suitable regularity class, we {\em formally} get the following candidate limiting operator
\begin{equation}\label{limiting_Hamiltonian:nu=4_case}
H^{(\nu=4)} f(r) = - \tfrac{1}{96} \Gamma^4(0) \left( 5\Gamma^{(4)}(0) - \beta \Gamma^{(5)}(0) \right) r^3 f'(r) + \tfrac{4\beta^2}{\Gamma(0)}r \left( f'(r) \right)^2.
\end{equation} 

To rigorously conclude that the Hamiltonian $H^{(\nu=\cdot)}$ is the limit of the sequence $\{H_n\}_{n \geq 1}$, given in \eqref{eqn:generic_Hamiltonian_expansion}, when $\nu$ is either $2$ or $4$, we have to prove that $H^{\nu = \cdot} \subseteq \LIM_n H_n$ (see Definition~\ref{def:definition_exLIM}).

\paragraph{Limiting Hamiltonian.} Note that, due to the change to polar coordinates, we have a singular behavior at the boundary point $r=0$. To repair for this singularity we need to define a suitable regularity class of functions on $\mathfrak{S}$.\\
Recall that $C_c^k(\bR^+ \! \times S^1)$ (resp. $C_{0}^k(\bR^+ \! \times S^1)$) denotes the set of functions on $\mathbb{R}^+ \! \times S^1$, of class $C^k$, that are constant (resp. zero) both on a neighbourhood of zero and on a neighbourhood of infinity.  More precisely, 
\begin{small}
\[
C_c^k(\bR^+ \!\! \times \! S^1) = \left\{f \in C_b^k(\bR^+ \!\! \times \! S^1) \, \middle| \, \exists \, c_1,c_2 \in \bR, \delta > 0: \, f(x) = c_1 \text{ for } \|x\| < \delta \wedge \delta^{-1}, f(x) = c_2 \text{ for } \|x\| > \delta \vee \delta^{-1} \right\}
\]
\end{small}
and
\begin{small}
\[
C_{0}^k(\bR^+ \! \times S^1) = \left\{f \in C_b^k(\bR^+ \! \times S^1) \, \middle| \, \exists \, \delta > 0: \, f(x) = 0 \text{ for } \|x\| < \delta \wedge \delta^{-1} \text{ and } \|x\| > \delta \vee \delta^{-1} \right\}.
\]
\end{small}
We introduce the following class of functions.

\begin{definition}
Let $k \geq 1$. Consider functions $f \in C_c^k(\bR^+)$ and  $\varphi \in C_c^k(\mathbb{R}^+ \! \times S^1)$.
We say that $F \in sC^k_b(\fS)$ if the function $F$ is of the form 
\begin{equation*}
F(r,\theta) = \begin{cases}
f(0) & \text{if } r = 0, \theta \in S^1 \\
f(r) + \varphi(r, \theta) & \text{if } r > 0, \theta \in S^1.
\end{cases} 
\end{equation*}
\end{definition}

\begin{lemma}\label{lemma:regularity_f_on_German_S}
Let $k \geq 1$ and let $\Phi$ be the map in \eqref{definition:phi}. If $F \in sC_b^k(\fS)$, then $F \circ \Phi \in C^k_b(\bR^2)$ and $F \circ \Phi$ is constant both on a neighbourhood of the origin and on a neighbourhood of infinity. 
\end{lemma}

\begin{proof}
Let $F \in sC_b^k(\fS)$. We show that $F \circ \Phi$ is $k$ times continuously differentiable. First note that by definition $F$ is constant on a neighbourhood of the origin. Thus, the first to $k$-th order derivatives at $r=0$ pose no problem. On the complement of this neighbourhood of the origin, we can use the chain rule to calculate derivatives, using that $\atan2$ is smooth. Since $F$ is constant outside some compact set, all derivatives are bounded.
\end{proof}

Observe that thanks to Lemma~\ref{lemma:regularity_f_on_German_S} we can give sense to the definition of the Hamiltonian \eqref{def:definition_generic_Hamiltonian} as $F \in \mathcal{D}(H_n)$ implies $F \circ \Phi \in \mathcal{D}(H_n)$.

We are now ready to define the sequence of functions along which we will be able to prove convergence of the Hamiltonians. With a slight abuse, we will use the same notation as in the previous section for our functions. For any $f \in C_c^2(\mathbb{R}^+)$, we define the perturbations 
\begin{equation}\label{perturbation:nu=2}
F_{n,f}^{(\nu=2)} (r,\theta) := 
\left\{
\begin{array}{ll}
f(0) & \text{ if } r=0, \theta \in S^1\\
f(r) + b_n^{-2} \Lambda_f^{(0)}(r,\theta) & \text{ if } r>0, \theta \in S^1
\end{array}
\right.
\end{equation}
and
\begin{equation}\label{perturbation:nu=4}
F_{n,f}^{(\nu=4)} (r,\theta) := 
\left\{
\begin{array}{ll}
f(0) & \text{ if } r=0, \theta \in S^1\\
f(r) + b_n^{-2} \Lambda_f^{(1)}(r,\theta) + b_n^{-4} \Lambda_f^{(2)}(r,\theta) & \text{ if } r>0, \theta \in S^1,
\end{array}
\right.
\end{equation}
where the functions $\Lambda_f^{(i)}$ ($i=0,1,2$) are defined in \eqref{definition_Lambda_0}, \eqref{definition_Lambda_1} and \eqref{definition_Lambda_2}, respectively. Observe that now $F_{n,f}^{(\nu = 2)} \in sC_b^2(\mathfrak{S})$ and $F_{n,f}^{(\nu = 4)} \in sC_b^2(\mathfrak{S})$.
%
%
We want to show that, for every $f \in C_c^{\infty}(\mathbb{R}^+)$, it holds
\begin{enumerate}[(I)]
\item \label{convergence:item_1}
$\LIM_n F_{n,f}^{(\nu=2)} = f$ and $\LIM_n F_{n,f}^{(\nu=4)} = f$;
\item \label{convergence:item_2}
$\LIM_n H_n F_{n,f}^{(\nu=2)} = H^{(\nu=2)} f$ and $\LIM_n H_n F_{n,f}^{(\nu=4)} = H^{(\nu=4)} f$.
\end{enumerate} 
The next lemma proves \eqref{convergence:item_1} and the convergence of the gradients.

\begin{lemma}\label{lmm:convergence_perturbations_and_gradients}
Suppose we are in the setting of Theorem~\ref{thm:crt:md:CWdiss} and $\nu=2$. Define the approximation $F_{n,f}^{(\nu=2)} \in sC_b^{\infty}(\mathfrak{S})$ as in \eqref{perturbation:nu=2}+\eqref{definition_Lambda_0}.
Moreover, let $Q := [0,q] \times S^1$, with $q \geq 0$, be a rectangle in $\mathfrak{S}$. Then, $F_{n,f}^{(\nu=2)} \in C_c^{\infty}(\mathfrak{S})$, $\LIM_n F_{n,f}^{(\nu=2)}=f$ and 
\begin{equation}\label{convergence_gradients}
\lim_{n \to \infty} \sup_{(r,\theta) \in Q} \left\vert \nabla F_{n,f}^{(\nu=2)}(r,\theta) - \nabla f(r) \right\vert = 0,
\end{equation}
for all rectangles $Q$ in $\mathfrak{S}$. An analogous statement holds true in the setting of Theorem~\ref{thm:tri-crt:md:CWdiss}, with $\nu = 4$, for the approximation $F_{n,f}^{(\nu = 4)} \in sC_b^{\infty}(\mathfrak{S})$ given by \eqref{perturbation:nu=4}+\eqref{definition_Lambda_1}+\eqref{definition_Lambda_2}. 
\end{lemma}

\begin{proof}
We prove only the $\nu=2$ case, the other being similar. If $r=0$ the assertions are obviously fulfilled, since $F_{n,f}^{(\nu=2)}(0,\theta)=f(0)$ for all $n \in \mathbb{N}$. Let us focus on the case $r>0$. Observe that $\Lambda_f^{(0)} \in C_c^{\infty}(\mathfrak{S})$ as composition of mappings $r \mapsto \Lambda_f^{(0)}(r,\cdot)$  of class $C_c^{\infty}(\mathbb{R}^+)$ and $\theta \mapsto \Lambda_f^{(0)}(\cdot, \theta)$ of class $C^{\infty}(S^1)$. As a consequence, $F_{n,f}^{(\nu=2)} \in C_c^{\infty}(\mathfrak{S})$ and hence it is uniformly bounded on $\mathbb{R}^2$. Moreover, as $b_n \to \infty$, we find that
\[
\lim_{n \to \infty} \, \sup_{(r,\theta) \in Q} \, \left\vert F_{n,f}^{(\nu=2)}(r,\theta) - f(r) \right\vert + \left\vert \begin{pmatrix} \partial_r F_{n,f}^{(\nu=2)}(r,\theta) \\[0.2cm] \partial_{\theta} F_{n,f}^{(\nu=2)}(r,\theta) \end{pmatrix} - \begin{pmatrix} f'(r) \\[0.2cm] 0 \end{pmatrix} \right\vert = 0,
\]
for all rectangles $Q$ in $\mathfrak{S}$. The second part of the limiting statement establishes \eqref{convergence_gradients}. Its first part together with uniform boundedness of the sequence $F_{n,f}^{(\nu=2)}$ gives $\LIM_n F_{n,f}^{(\nu=2)}=f$.
\end{proof}

For the proof of \eqref{convergence:item_2}, we start by characterizing the expression of the limiting Hamiltonian and checking that we can control the remainders in the expansions. Then, we prove the operator convergence.

\begin{lemma}\label{lmm:limiting_Hamiltonian:controlled_remainder}
Suppose we are either in the setting of Theorem~\ref{thm:crt:md:CWdiss} and $\nu=2$ or in the setting of Theorem~\ref{thm:tri-crt:md:CWdiss} and $\nu=4$. Consider the sequence $\{H_n\}_{n \geq 1}$ where, for each $n \in \mathbb{N}$, the Hamiltonian $H_n$ is given by \eqref{eqn:generic_Hamiltonian_expansion}. Define the perturbation $F_{n,f}^{(\nu=\cdot)} \in sC_b^{\infty}(\mathfrak{S})$ as in \eqref{perturbation:nu=2}+\eqref{definition_Lambda_0} if $\nu=2$ or as in \eqref{perturbation:nu=4}+\eqref{definition_Lambda_1}+\eqref{definition_Lambda_2} if $\nu=4$. \\
Then, we obtain
\begin{equation}\label{limiting_Hamiltonian:nu=dot_case}
H_n F_{n,f}^{(\nu=\cdot)} (r,\theta) = H^{(\nu=\cdot)} f(r) + R_n^{(\nu = \cdot)}(r,\theta),
\end{equation}
with $H^{(\nu = \cdot)}$ as in \eqref{limiting_Hamiltonian:nu=2_case} if $\nu=2$ or as in \eqref{limiting_Hamiltonian:nu=4_case} if $\nu=4$. Moreover, we have $\sup_n \|H_n F_{n,f}^{(\nu=\cdot)}\| < \infty$ and, as $n \to \infty$, the remainder term $R_n^{(\nu=\cdot)}$ converges to zero uniformly.
\end{lemma}

\begin{proof}
A straightforward computation gives \eqref{limiting_Hamiltonian:nu=dot_case} (apply \eqref{eqn:generic_Hamiltonian_expansion} to either \eqref{perturbation:nu=2} or \eqref{perturbation:nu=4}). Moreover, since $f \in C_c^\infty(\mathbb{R}^+)$, we have $\vn{H^{(\nu=\cdot)}f} < \infty$. We are left to analyze the remainder term.\\
Let $\Phi$ be the map in \eqref{definition:phi}. Observe that the remainder $R_n^{(\nu=\cdot)}$ includes all the remainder terms coming from the Taylor expansions of the functions $\Gamma$ and the exponentials in \eqref{eqn:generic_Hamiltonian:first_expansion}. We study the Taylor expansion of the exponential functions first. We treat explicitly only the case of
\begin{equation}\label{control_remainder:term_1}
\frac{b_n^{2\nu+2}}{2} \left[ \exp \left\{\frac{n}{b_n^{\nu+2}} \left[ \left(F_{n,f}^{(\nu=\cdot)} \circ \Phi \right)\left( x, \xi + \frac{2\beta b_n}{\Gamma(0)n} \right) - \left(F_{n,f}^{(\nu=\cdot)} \circ \Phi \right)(x,\xi) \right] \right\} - 1 \right],
\end{equation}
the other being analogous. We denote by $R_{n,\mathrm{exp}}^{(\nu=\cdot)}$ the remainder terms coming from expanding \eqref{control_remainder:term_1}. To shorten our next formulas, we drop all superscripts ``$(\nu=\cdot)$'' from the perturbations, we define $\mathbf{x} := (x,\xi)$, $\mathbf{y} := (x,\eta)$ and we set $c_0 := \frac{\beta}{\Gamma(0)}$. By Lagrange's form of the Taylor expansion, we can find a $\mathbf{y} \in \mathfrak{S}$ with $\eta \in \left( \xi, \xi+ 2c_0 b_n n^{-1} \right)$ and
\begin{multline*}
R_{n,\mathrm{exp}}^{(\nu=\cdot)}(\mathbf{x}) = c_0^2 \,  b_n^{\nu+2} n^{-1} \, \partial_2^2 \left( F_{n,f} \circ \Phi \right)(\mathbf{x}) \\
+ \tfrac{2}{3} c_0^3 \exp \left\{ b_n^{-(\nu+2)} n \left[ \left( F_{n,f} \circ \Phi \right)(\mathbf{y}) - \left( F_{n,f} \circ \Phi \right)(\mathbf{x}) \right] \right\}  \bigg\{ b_n^{-(\nu+1)} \, \left( \partial_2 \left( F_{n,f} \circ \Phi \right)(\mathbf{y}) \right)^3 \\
+ 3 b_n n^{-1} \, \partial_2 \left( F_{n,f} \circ \Phi \right)(\mathbf{y}) \, \partial_2^2 \left( F_{n,f} \circ \Phi \right)(\mathbf{y}) +  b_n^{\nu+3} n^{-2} \, \partial_2^3 \left( F_{n,f} \circ \Phi \right)(\mathbf{y}) \bigg\}.
\end{multline*}
By the mean-value theorem, we can control the exponential in the previous display. Indeed, there exists a point $\mathbf{z} \in \mathfrak{S}$, on the line-segment connecting $\mathbf{x}$ and $\mathbf{y}$, for which we have 
\[
\left(F_{n,f} \circ \Phi \right)(\mathbf{y}) - \left(F_{n,f} \circ \Phi \right) (\mathbf{x}) = \ip{\nabla \left(F_{n,f} \circ \Phi \right)(\mathbf{z})}{\mathbf{y} - \mathbf{x}}.
\]
Since $\mathbf{y} - \mathbf{x} \in \,  \{0\} \times \left(0, 2c_0 b_n n^{-1} \right)$, we can estimate 
\[
|\left(F_{n,f} \circ \Phi \right)(\mathbf{y}) - \left(F_{n,f} \circ \Phi \right) (\mathbf{x})| \leq 2 c_0 b_n n^{-1} \|\partial_2 \left(F_{n,f} \circ \Phi\right) \|) 
\]
and, in turn, we get 
\begin{align*}
\exp \left\{ nb_n^{-(\nu+2)} \left[\left(F_{n,f} \circ \Phi \right)(\mathbf{y}) - \left(F_{n,f} \circ \Phi \right) (\mathbf{x}) \right] \right\} &\leq \exp \left\{ 2c_0 b_n^{-(\nu+1)} \left\| \partial_2 \left( F_{n,f} \circ \Phi \right) \right\| \right\} \\
&\leq \exp \left\{ 2c_0 \left\| \partial_2 \left( F_{n,f} \circ \Phi \right) \right\| \right\}.
\end{align*}
Recall that by assumption $b_n^{\nu+2} n^{-1} \to 0$, as $n \to \infty$. We can then find positive constants $c_1$ and $c_2$ (depending on the sup-norms of the first, second and third order partial derivatives of $F_{n,f} \circ \Phi$, and, therefore, on the sup-norms of the first four partial derivatives of $f \circ \Phi$ and, hence, not on $n$), such that 
\[
\sup_{(x,\xi) \in \mathfrak{S}} \left\vert R_{n,\mathrm{exp}}^{(\nu=\cdot)}(x,\xi) \right\vert \leq c_1 \,  b_n^{\nu+2} n^{-1} + c_2 \, b_n^{-(\nu+1)}.
\]
%
We focus now on the remainder terms relative to the expansion of the function $\Gamma$. We have
\[
R_{n,\Gamma}^{(\nu=\cdot)}(\xi) = \frac{2 b_n^{-1} \beta \Gamma^{j}(0) \Gamma^{(j+1)}(\zeta)}{(j+1)!} \, \xi^{j+1} \, \partial_{\xi} \left( F_{n,f} \circ \Phi \right)(x,\xi), \qquad \, (j=3, \text{ if } \nu=2; \, j=5, \text{ if } \nu=4),
\] 
with $\zeta \in \, (0,\Gamma(0)\xi b_n^{-1})$. Since $F_{n,f} \circ \Phi$ is compactly supported (cf. Lemma~\ref{lemma:regularity_f_on_German_S}), $\xi$ is bounded on $\mathbb{R}$ and we derive the following bound
\[
\sup_{\xi \in \mathbb{R}} \left\vert R_{n,\Gamma}^{(\nu=\cdot)}(\xi) \right\vert \leq c_3 \, b_n^{-1},
\]
where $c_3$ is a suitable positive constant, independent of $n$. Analogous estimates hold for the second term in \eqref{eqn:generic_Hamiltonian:first_expansion}. Putting everything together, we get
\[
\sup_{(r,\theta) \in \mathfrak{S}} \left\vert R_{n}^{(\nu=\cdot)}(r,\theta) \right\vert \leq 2 \left[ c_1 \,  b_n^{\nu+2} n^{-1} + c_2 \, b_n^{-(\nu+1)} +c_3 \, b_n^{-1} \right],
\]
from which the conclusion follows.
\end{proof}

\begin{proposition}\label{prop:convergence_Hamiltonian_criticality}
Suppose we are either in the setting of Theorem~\ref{thm:crt:md:CWdiss} and $\nu=2$ or in the setting of Theorem~\ref{thm:tri-crt:md:CWdiss} and $\nu=4$. Consider the sequence $\{H_n\}_{n \geq 1}$ where, for each $n \in \mathbb{N}$, the Hamiltonian $H_n$ is given by \eqref{eqn:generic_Hamiltonian_expansion}. Moreover, consider the Hamiltonian $(H,C_c^{\infty}(\mathbb{R}^+))$ of the type $Hf(r)=H(r,f'(r))$ with 
\begin{enumerate}[(a)]
\item
in the setting of Theorem~\ref{thm:crt:md:CWdiss} and $\nu=2$:
\[
H(r,p) = - \tfrac{1}{4} \Gamma^2(0) \left( 3 \Gamma''(0) - \beta \Gamma'''(0)\right) r^2 p + \tfrac{4\beta^2}{\Gamma(0)} r p^2;
\]
\item 
in the setting of Theorem~\ref{thm:tri-crt:md:CWdiss} and $\nu=4$:
\[
H(r,p) = - \tfrac{1}{96} \Gamma^4(0) \left( 5 \Gamma^{(4)}(0) - \beta \Gamma^{(5)}(0) \right) r^3 p + \tfrac{4\beta^2}{\Gamma(0)} rp^2.
\]
\end{enumerate}
Then we have $H \subseteq ex-\LIM_n H_n$.
\end{proposition}

\begin{proof}
Consider the setting (a). Fix $f \in C_c^{\infty}(\mathbb{R}^+)$ and let $F_{n,f}^{(\nu=2)}$ be the approximation defined by \eqref{perturbation:nu=2}+\eqref{definition_Lambda_0}. Lemma~\ref{lmm:convergence_perturbations_and_gradients}  gives $\LIM_n F_{n,f}^{(\nu=2)} = f$ and the uniform convergence of the gradients.  Therefore, by Lemma~\ref{lmm:limiting_Hamiltonian:controlled_remainder} we can indeed conclude  $\LIM_n H_n F_{n,f}^{(\nu=2)} = Hf$.  The proof in the case (b) is analogous. 
\end{proof}

\subsubsection{Exponential compact containment.}
\label{subsect:exponential_compact_containment}

We must verify exponential compact condition for the fluctuation process. The validity of the compactness condition will be shown in Proposition~\ref{prop:exponential_compact_containment}. We start by proving an auxiliary lemma. 

\begin{lemma}\label{lmm:uniform_containment_bound}
	Suppose we are either in the setting of Theorem~\ref{thm:crt:md:CWdiss} and $\nu=2$ or in the setting of Theorem~\ref{thm:tri-crt:md:CWdiss} and $\nu=4$. Let $G = [0,q) \times S^1$ and let $f_\star \in C_c^\infty(\bR^+)$ be such that $f_\star(r) = \log(1+r)$ for $r \in [1,q]$ and $0 \leq f_\star'(r) \leq (1+r)^{-1}$.	Define
	\[
	\Upsilon_n (r,\theta) := F_{n,f_{\star}}^{(\nu=\cdot)}(r,\theta),
	\]
	with $F_{n,f_{\star}}^{(\nu=\cdot)}$ as in \eqref{perturbation:nu=2}+\eqref{definition_Lambda_0} if $\nu=2$ or as in \eqref{perturbation:nu=4}+\eqref{definition_Lambda_1}+\eqref{definition_Lambda_2} if $\nu=4$. We have
	\[
	\limsup_{n \to \infty} \sup_{(r,\theta) \in G} H_n \Upsilon_n (r,\theta) \leq \frac{4\beta^2}{\Gamma(0)}.
	\]
\end{lemma}

\begin{proof}
	Consider the setting of Theorem~\ref{thm:crt:md:CWdiss} and $\nu=2$. By \eqref{limiting_Hamiltonian:nu=2_case} if $r \in [1,q]$, we have that 
	\[
	H_n F_{n,f_{\star}}^{(\nu=2)}(r,\theta) = -\tfrac{1}{4}(3 \Gamma''(0)-\beta \Gamma'''(0)) \tfrac{r^2}{1+r} + \tfrac{4\beta^2}{\Gamma(0)} \tfrac{r}{(1+r)^2} + o(1).
	\]
	By Lemma \ref{lmm:limiting_Hamiltonian:controlled_remainder} we find that, as $f_{\star} \in C_c^{\infty}(\mathbb{R}^+)$, the remainder $o(1)$ includes all terms dominated by $cb_n^{-2}$, for a suitable positive constant $c$ (independent of $n$). Using that $3 \Gamma''(0)-\beta \Gamma'''(0) \geq 0$, $0 \leq f_\star'(r) \leq (1+r)^{-1}$, and that the mapping $r \mapsto \frac{r}{(1+r)^2}$ is bounded, we obtain
	\[
	H_n F_{n,f_{\star}}^{(\nu=2)}(r,\theta) \leq \tfrac{4\beta^2}{\Gamma(0)} + cb_n^{-2}, 
	\]
for all $r \in [0,q)$, from which the conclusion follows. The proof in the setting of Theorem~\ref{thm:tri-crt:md:CWdiss}, with $\nu=4$, is analogous and gives the same bound. 
\end{proof}

\begin{proposition}\label{prop:exponential_compact_containment}
Suppose we are either in the setting of Theorem~\ref{thm:crt:md:CWdiss} and $\nu=2$ or in the setting of Theorem~\ref{thm:tri-crt:md:CWdiss} and $\nu=4$. Moreover, assume that $\left(R_n(0), \Theta_n(0) \right)$ is exponentially tight at speed $nb_n^{-\nu-2}$, then the process $\left\{\left(R_n(b_n^{\nu}t), \Theta_n(b_n^{\nu} t)\right)\right\}_{t \geq 0}$ satisfies the exponential compact containment condition at speed $n b_n^{-\nu-2}$. 
\end{proposition}

\begin{proof}
The statement follows from Lemmas \ref{lmm:uniform_containment_bound} and \ref{lemma:compact_containment_FK} by choosing $f_n \equiv \Upsilon_n$ on a fixed, sufficiently large, compact set of $\mathbb{R}^2$. For similar proofs see e.g. \cite[Lem. 3.2]{DFL11} or \cite[Prop. A.15]{CoKr17}.
\end{proof}

\emph{Proof of Theorems~\ref{thm:crt:md:CWdiss} and \ref{thm:tri-crt:md:CWdiss}.} We check the assumptions of Theorem~\ref{theorem:Abstract_LDP}. Assumption~\ref{assumption:LDP_assumption} follows from Proposition~\ref{prop:convergence_Hamiltonian_criticality}. The comparison principle for $f-\lambda Hf = h$, for $h \in C_b(\mathbb{R}^+)$ and $\lambda > 0$, will be proved in Section~\ref{section:comparison_principle_singular_hamiltonian}. Finally, the exponential compact containment has been verified in Proposition~\ref{prop:exponential_compact_containment}.

\section{The comparison principle for singular Hamiltonians} \label{section:comparison_principle_singular_hamiltonian}

In \cite[App.~A]{CoKr17}, a proof of the comparison principle is given for a broad class of Hamilton-Jacobi equations on subsets of $\bR^d$. The Hamiltonians on $\bR^+$ considered in this paper technically fall within the scope of that appendix, except for the fact that the behaviour at the boundary point $0$ is singular. A related Hamilton-Jacobi equation is considered in \cite{DFL11}, but in that setting the boundary point is natural and so can be naturally excluded from the state space.\\  
This section focuses on the treatment of the comparison principle  for a class of Hamiltonians with a singular boundary point  included in the state space. Our setting is as follows.

\begin{assumption} \label{assumption:abstract_comparison_principle}
The Hamiltonian $H \subseteq C(\bR^+) \times C(\bR^+)$ has domain $\cD(H) = C_c^2(\mathbb{R}^+)$ and, for $f \in C_{c}^2(\bR^+)$, it is of the form $Hf(x) = H(x,f'(x))$ with
\begin{equation}\label{eqn:generic_limiting_Hamiltonian}
H(x,p) = - b x^k p + a x p^2, \qquad \qquad (a>0; b \geq 0; k \in [1,\infty)).
\end{equation}
\end{assumption}

We will stay close to the ideas introduced in \cite{FK06, DFL11}, which were also explained in \cite{CoKr17}.  For the verification of the comparison principle two types of functions are of importance: good penalization and good containment functions.

\begin{definition}
Let $\alpha > 0$ and consider $\Psi_\alpha : \mathbb{R}^+ \times \mathbb{R}^+ \rightarrow \bR$.  We say that the family $\{\Psi_\alpha\}_{\alpha >0}$ is a collection of \textit{good penalization functions} (for $H$) if
\begin{enumerate}[($\Psi$a)]
\item For all $\alpha > 0$, we have $\Psi_\alpha \geq 0$ and $\Psi_\alpha(x,y) = 0$ if and only if $x = y$. Additionally, $\alpha \mapsto \Psi_\alpha$ is increasing and
\begin{equation*}
\lim_{\alpha \rightarrow \infty} \Psi_\alpha(x,y) = \begin{cases}
0 & \text{if } x = y \\
\infty & \text{if } x \neq y.
\end{cases}
\end{equation*}	
\item $\Psi_\alpha$ is twice continuously differentiable on $(0,\infty)^2$ for all $\alpha > 0$.
\end{enumerate}
\end{definition}

\begin{definition} \label{definition:good_containment}
Let $H \subseteq C_b(\bR^+) \times C_b(\bR^+)$ with $\cD(H) \subseteq C^1(\bR^+)$ of the type $Hf(x) = \cH(x,f'(x))$, where $(x,p) \mapsto \cH(x,p)$ is continuous. We say that $\Upsilon : \bR^+ \rightarrow \bR$ is a \textit{good containment function} (for $H$) if
\begin{enumerate}[($\Upsilon$a)]
\item $\Upsilon \geq 0$ and there exists a point $x_0 \in \bR^+$ such that $\Upsilon(x_0) = 0$,
\item $\Upsilon$ is twice continuously differentiable, 
\item for every $c \geq 0$, the set $\{x \in E \, | \, \Upsilon(x) \leq c\}$ is compact,
\item we have $\sup_{x \geq 0} \cH(x,\nabla \Upsilon(x)) < \infty$.
\end{enumerate}
\end{definition}

To conclude the proof of our moderate deviation principles we have to show that there exists a good containment function $\Upsilon$ for $H$ and that the comparison principle holds for the Hamilton-Jacobi equation $f - \lambda Hf = h$ (with $\lambda >0$; $h \in C_b(\mathbb{R}^+)$). We immediately give a containment function that can be used in our setting.

\begin{lemma} \label{lemma:containment_function_dCW_speedup}
Let Assumption \ref{assumption:abstract_comparison_principle} be satisfied. The function $\Upsilon(x) = \log \left(1 + x\right)$ is a good containment function for $H$.
\end{lemma}

\begin{proof} 
We have $\Upsilon \in C^2(\mathbb{R}^+)$, $\Upsilon(0) = 0$, $\lim_{x \rightarrow +\infty} \Upsilon(x) = +\infty$, and finally, we have
\begin{equation*}
\cH(x,\nabla \Upsilon(x)) = -b \frac{x^k}{1+x} + a \frac{x}{(1+x)^2}, 
\end{equation*}
which is uniformly bounded from above in $x$ as $b \geq 0$ and the mapping $x \mapsto \frac{x}{(1+x)^2}$ is bounded.
\end{proof}

A second aspect to be taken into account in the study of our comparison principle are good penalization functions. Differently to \cite[App.~A]{CoKr17}, a direct application of such functions is not possible in the present setting, due to the singularity at $x=0$. However, a similar structure is present. We will exploit such a structure and give an ad-hoc treatment. The underlying principle was introduced in \cite{DFL11} and further used in \cite{KrReVe18}: the penalization function should equal the square of the Riemannian metric generated by the quadratic part of the Hamiltonian. The quadratic part in our context generates a squared distance $d^2(x,y) = (\sqrt{x} - \sqrt{y})^2$, which is not differentiable at $x=0$ or $y=0$. Nevertheless, it can be shown that the uniform closure of the Hamiltonian does contain perturbations of $d^2$. This implies we can use the istance function $d$ as a good penalization function and follow \cite{CoKr17}. 

\begin{lemma} \label{lemma:square_root_distance_in_domain}
Let Assumption \ref{assumption:abstract_comparison_principle} be satisfied. The uniform closure $\overline{H}$ of $H$ contains functions $f \in C_b(\bR^+)$ satisfying
\begin{enumerate}[(a)]
\item There are constants $c \in \bR$ and $M \geq 0$ such that $f(x) = c$ for $x \geq M$.
\item The restriction of $f$ to $(0,\infty)$ is twice continuously differentiable.
\item There exists a twice continuously differentiable function $\hat{f}$ on $[0,M + 1]$ and two constants $c_1,c_2 \in \bR$ such that $f(x) = c_1(\sqrt{x} - c_2)^2 + \hat{f}(x)$ for all $x \in [0,M+1]$.
utions\end{enumerate}
For $f$ of this type, we have
\begin{equation} \label{eqn:lemma_closure_H_with_sqrt_distance}
\begin{aligned} 
\overline{H}f(x) & = \begin{cases}
\cH(x,\nabla f(x)) & \text{if } x \neq 0, \\
a c_1^2 c_2^2 & \text{if } x = 0,
\end{cases} \\
& = \begin{cases}
-b\left[c_1(\sqrt{x} - c_2) x^{k-1} \sqrt{x} + x^k \hat{f}'(x)\right] + a \left[c_1(\sqrt{x} - c_2) + \sqrt{x} \hat{f}'(x)\right]^2 & \text{if } x \leq M, \\
0 & \text{if } x > M,
\end{cases} 
\end{aligned}
\end{equation}
which is a bounded continuous function on $\mathbb{R}^+$.
\end{lemma}

\begin{proof}
Let $f$ satisfy the conditions (a)-(c) in the statement. Without loss of generality, assume $M \geq 2$. We construct approximations $f_n$ (of $f$) such that $\vn{f_n - f} + \vn{H f_n - g} = 0$, with $g$ of the form \eqref{eqn:lemma_closure_H_with_sqrt_distance}. 

For each $n \geq 1$, let $\rho_n : \bR^+ \rightarrow [0,1]$ be a smooth monotone function satisfying
\begin{equation*}
\rho_n(x) = \begin{cases}
0 & \text{if } x \leq 2^{-n}, \\
1 & \text{if } x \geq 2^{-n + 1}.
\end{cases}
\end{equation*}
Note that for $x \in (0,M+1)$, we have
\begin{equation*}
f'(x) = \frac{c_1\sqrt{x} - c_1c_2}{\sqrt{x}} + \hat{f}'(x).
\end{equation*}
We define functions $f_n$ such that $f_n(x) = f(x)$ for $x \geq M$ and such that for $x \leq M$, we have $f_n(x) = \hat{f}_n(x) + \varphi_n(x)$ with $\hat{f}_n(M) = \hat{f}(M)$, $\varphi_n(M) = c_1(\sqrt{M} - c_2)^2$ and
\begin{equation*}
\hat{f}_n'(x) = \rho_n(x)\hat{f}'(x), \qquad \varphi_n'(x) = \rho_n(x)\frac{c_1\sqrt{x} - c_1c_2}{\sqrt{x}}. 
\end{equation*}
As $\varphi_n'$ has constant sign, $\varphi_n$ converges uniformly on $[0,M]$. As $\hat{f}$ is bounded, $\hat{f}_n$ converges uniformly to $\hat{f}$ on $[0,M]$. To prove that $Hf_n$ converges uniformly to $g$, note that for $x \geq M$, the sequence is constant and equal to its limiting value as in \eqref{eqn:lemma_closure_H_with_sqrt_distance}. For $x \leq M$, we first calculate that
\begin{align*}
x^k f_n'(x) & = \rho_n(x)\left(c_1\sqrt{x} - c_1c_2\right) x^{k-1} \sqrt{x} + \rho_n(x) x^k \hat{f}'(x), \\
x (f_n'(x))^2 & = \rho_n(x)^2 \left(c_1\sqrt{x} - c_1c_2 + \sqrt{x}\hat{f}'(x)\right)^2.
\end{align*}
We immediately obtain that $x^k f_n'(x)$ converges to $c_1 x^k - c_1c_2 x^{k-1/2} + x^k \hat{f}'(x)$, uniformly on $[0,M]$, and that $x(f_n'(x))^2$ converges to $\left(c_1\sqrt{x} - c_1c_2 + \sqrt{x}\hat{f}'(x)\right)^2$, uniformly on $[0,M]$.  Combining these statements we find that $\lim_{n \rightarrow \infty} \vn{H_nf - g} = 0$ and that $g$ has the form \eqref{eqn:lemma_closure_H_with_sqrt_distance}.
\end{proof}

We introduce two convenient viscosity extensions of the Hamiltonian $H$ in terms of the penalization functions $\{\Psi_\alpha\}_{\alpha > 0}$, with $\Psi_{\alpha}(x,y) = \alpha\left(\sqrt{x} - \sqrt{y}\right)^2$, and the containment function $\Upsilon$. Let
\begin{align*}
\cD(H_\dagger) & := \left\{x \mapsto  (1-\varepsilon)\Psi_\alpha(x,y) + \varepsilon \Upsilon(x) +c \, \middle| \, \alpha,\varepsilon > 0, c \in \bR \right\} \\
\cD(H_\ddagger) & := \left\{y \mapsto - (1+\varepsilon)\Psi_\alpha(x,y) - \varepsilon \Upsilon(y) +c \, \middle| \, \alpha,\varepsilon > 0, c \in \bR \right\}
\end{align*}
and, for $f \in \cD(H_\dagger)$ (resp. $f \in \cD(H_\ddagger)$), define
\[
H_\dagger f(x) = \begin{cases} \cH(x,\nabla f(x)) & \text{if } x \neq 0 \\
a(1-\varepsilon)^2 \alpha^2 y   & \text{if } x = 0
\end{cases} 
\quad \text{ and } \quad
H_\ddagger f(x) = \begin{cases} \cH(x,\nabla f(x)) & \text{if } x \neq 0\\
a(1+\varepsilon)^2 \alpha^2 y  & \text{if } x = 0.
\end{cases}
\]
where $H$ is given in \eqref{eqn:generic_limiting_Hamiltonian}.

\begin{lemma} \label{lemma:viscosity_extension}
The operator $(H_\dagger,\cD(H_\dagger))$ is a viscosity sub-extension of $H$ and $(H_\ddagger,\cD(H_\ddagger))$ is a viscosity super-extension of $H$.
\end{lemma}

In the proof we need \cite[Lem.~7.7]{FK06}. We recall it here for the sake of readability. Let $M_\infty(E,\overline{\bR})$ denote the set of measurable functions $f : E \rightarrow \bR \cup \{\infty\}$ that are bounded from below.
\begin{lemma}[Lem. 7.7 in \cite{FK06}] \label{lemma:extension_lemma_7.7inFK}
Let $B$ and $B_\dagger \subseteq M_\infty(E,\overline{\bR}) \times M(E,\overline{\bR})$ be two operators. Suppose that for all $(f,g) \in B_\dagger$ there exist $\{(f_n,g_n)\} \subseteq B_\dagger$ that satisfy the following conditions:
\begin{enumerate}[(a)]
\item For all $n$, the function $f_n$ is lower semi-continuous.
\item For all $n$, we have $f_n \leq f_{n+1}$ and $f_n \rightarrow f$ point-wise.
\item Suppose $x_n \in E$ is a sequence such that $\sup_n f_n(x_n) < \infty$ and $\inf_n g_n(x_n) > - \infty$, then $\{x_n\}_{n \geq 1}$ is relatively compact and if a subsequence $x_{n(k)}$ converges to $x \in E$, then
\begin{equation*}
\limsup_{k \rightarrow \infty} g_{n(k)}(x_{n(k)}) \leq g(x).
\end{equation*}
\end{enumerate}
Then $B_\dagger$ is a viscosity sub-extension of $B$.\\
An analogous result holds for super-extensions $B_{\ddagger}$ by taking $f_n$ a decreasing sequence of upper semi-continuous functions and by replacing requirement (c) with
\begin{enumerate}
\item[(c$^{\prime}$)] Suppose $x_n \in E$ is a sequence such that $\inf_n f_n(x_n) > - \infty$ and $\sup_n g_n(x_n) <  \infty$, then $\{x_n\}_{n \geq 1}$ is relatively compact and if a subsequence $x_{n(k)}$ converges to $x \in E$, then
\begin{equation*}
\liminf_{k \rightarrow \infty} g_{n(k)}(x_{n(k)}) \geq g(x).
\end{equation*}
\end{enumerate} 
\end{lemma}

\begin{proof}[Proof of Lemma \ref{lemma:viscosity_extension}]
First of all, observe that the uniform closure $\overline{H}$ is a viscosity extension of $H$. Next, we prove that $H_\dagger$ is a viscosity sub-extension of $\overline{H}$ by using Lemma \ref{lemma:extension_lemma_7.7inFK} for $B = \overline{H}$ and $B_\dagger = H_\dagger$. The proof that $H_\ddagger$ is a super-extension of $\overline{H}$ is similar and therefore omitted.

\smallskip

Consider a collection of smooth increasing functions $\chi_n: \bR \rightarrow \bR$ such that $0 \leq \chi_n'(x) \leq 1$ and  
\[
\chi_n(x) = 
\left\{
\begin{array}{ll}
x & \text{ if } x \leq n \\
n+1 & \text{ if } x \geq n+2.
\end{array}
\right.
\]
Note that $\chi_{n + 1} \geq \chi_n$ for all $n$.  Let $f \in M_{\infty}(E, \overline{\mathbb{R}})$ and set $f_n = \chi_n \circ f$. Clearly, $f_n$ is lower semi-continuous for all $n$, giving Lemma~\ref{lemma:extension_lemma_7.7inFK}(a). Lemma~\ref{lemma:extension_lemma_7.7inFK}(b) is a consequence of the fact that the map $n \mapsto f_n(x)$ is increasing for all $x \in \bR^+$.

Now, we verify Lemma \ref{lemma:extension_lemma_7.7inFK}(c). Pick any sequence $x_n \in \bR^+$ such that $\sup_n f_n(x_n) < \infty$ and $\inf_n g_n(x_n) > - \infty$. As $\Upsilon$ is a good containment function, $\alpha \geq 0$ and $\Psi_\alpha \geq 0$, we find that $\{x_n\}_{n \geq 1}$ is relatively compact. Without loss of generality, assume that $x_n$ converges to $x_0$. We have to prove that 
\begin{equation} \label{eqn:sub_extension_explicit_generator_bound}
\limsup_n \overline{H} f_n(x_n) \leq H_\dagger f(x_0)
\end{equation}
Suppose first that $x_0 = 0$. Pick some $N > (1-\varepsilon) \alpha y$. Then, for all $n \geq N$ and $x$ in a neighbourhood  of $0$, we have that $f_n(x) = f(x)$ and $\overline{H}f_n(x) = H_\dagger f(x)$. This immediately establishes \eqref{eqn:sub_extension_explicit_generator_bound}.

If $x_0 \neq 0$, there is some $N$ such that for $n \geq N$, we have $x_n > \frac{1}{2} x_0$. By construction, we have $f'_n(x_n) \rightarrow f'(x_0)$. As the function $(x,p) \mapsto H(x,p)$ is continuous on $\bR^+ \times \bR$ and $\overline{H}f(x) = H(x,f'(x))$ if $x > 0$, we conclude that $\lim_n \overline{H} f_n(x_n) = H_\dagger f(x_0)$ establishing \eqref{eqn:sub_extension_explicit_generator_bound}.
\end{proof}

\begin{proposition} \label{proposition:abstract_comparison_principle}
Let Assumption \ref{assumption:abstract_comparison_principle} be satisfied. The comparison principle holds for sub- and super-solutions to $f - \lambda H f = h$, for all $h \in C_b(\mathbb{R}^+)$ and $\lambda > 0$. 
\end{proposition}

\begin{proof}
We follow arguments similar to those in \cite[Prop.~3.5 and App.~A]{CoKr17}. For every $\alpha,\varepsilon >0$ let $x_{\alpha,\varepsilon},y_{\alpha,\varepsilon} \in \bR^+$ be such that
\begin{multline*} 
\frac{u(x_{\alpha,\varepsilon})}{1-\varepsilon} - \frac{v(y_{\alpha,\varepsilon})}{1+\varepsilon} -  \Psi_\alpha(x_{\alpha,\varepsilon},y_{\alpha,\varepsilon}) - \frac{\varepsilon}{1-\varepsilon}\Upsilon(x_{\alpha,\varepsilon}) -\frac{\varepsilon}{1+\varepsilon}\Upsilon(y_{\alpha,\varepsilon}) \\
= \sup_{x,y \in \bR^+} \left\{\frac{u(x)}{1-\varepsilon} - \frac{v(y)}{1+\varepsilon} - \Psi_\alpha(x,y)  - \frac{\varepsilon}{1-\varepsilon}\Upsilon(x) - \frac{\varepsilon}{1+\varepsilon}\Upsilon(y)\right\}.
\end{multline*}
We use the good containment function $\Upsilon(x) = \log(1+x)$ (cf. Lemma~\ref{lemma:containment_function_dCW_speedup}) and the collection of penalization functions $\{\Psi_{\alpha}\}_{\alpha > 0}$, with $\Psi_{\alpha}(x,y) = \alpha \left( \sqrt{x} - \sqrt{y}\right)^2$. Analogously to the result in \cite[Prop.~A.11]{CoKr17}, we find that the comparison principle is satisfied if we can verify that
\begin{equation}\label{condH:negative:liminf}
\liminf_{\varepsilon \rightarrow 0} \liminf_{\alpha \rightarrow \infty} 
\left(H\Psi_\alpha(\cdot,y_{\alpha,\varepsilon})\right)(x_{\alpha,\varepsilon}) - \left(H(-\Psi_\alpha(x_{\alpha,\varepsilon},\cdot))\right)(y_{\alpha,\varepsilon}) \leq 0.
\end{equation}
Compared to \cite{CoKr17}, the main adjustments in the proof are that we need viscosity sub- and super-extensions of $H$ built out of the good penalization functions $\Psi_\alpha$ taking into account singular behaviour at $x=0$. This adaptation was carried out in Lemma \ref{lemma:viscosity_extension}. Notice that the arguments based on the convexity of $p \mapsto H(x,p)$, given   in the proof of \cite[Prop.~A.11]{CoKr17}, do not change for $x = 0$. This can be verified directly using the formulas for $H_\dagger, H_\ddagger$ when $x$ or $y$ equal $0$.

\smallskip

We proceed with the proof of the comparison principle by verifying \eqref{condH:negative:liminf}. Fix $\varepsilon > 0$. For notational convenience, we drop the subscript $\varepsilon$ from the sequences $x_{\alpha,\varepsilon}$ and $y_{\alpha,\varepsilon}$.

By the form of the Hamiltonians $H_\dagger, H_\ddagger$ found in Lemma \ref{lemma:square_root_distance_in_domain}, we have
\begin{align*}
\left(H\Psi_\alpha(\cdot,y_{\alpha})\right)(x_{\alpha}) - \left(H(-\Psi_\alpha(x_{\alpha},\cdot))\right)(y_{\alpha}) & = - b \alpha x_\alpha^{k-1}\sqrt{x_\alpha}(\sqrt{x_\alpha} - \sqrt{y_\alpha}) + a \alpha^2 (\sqrt{x_\alpha} - \sqrt{y_\alpha})^2 \\
& \quad - \left( - b \alpha y_\alpha^{k-1}\sqrt{y_\alpha}  (\sqrt{x_\alpha} - \sqrt{y_\alpha}) - a \alpha^2 (\sqrt{x_\alpha} - \sqrt{y_\alpha})^2\right) \\
& = b \alpha (y_\alpha^{k-1}\sqrt{y_\alpha} - x_\alpha^{k-1}\sqrt{x_\alpha})(\sqrt{x_\alpha} - \sqrt{y_\alpha}) \\
& \leq 0
\end{align*}
as $b \geq 0$. Thus, \eqref{condH:negative:liminf} is trivially satisfied and the comparison principle holds for $f - \lambda H f = h$. 
\end{proof}

\begin{remark}
Note that the final bound on the difference of the two Hamiltonians is essentially saying that the drift in the Hamiltonians is one-sided Lipschitz with respect to the Riemannian metric generated by the quadratic part of the Hamiltonian. In this sense, this result is analogous to \cite[Prop.~3.5]{CoKr17}. See also \cite{DFL11} for a short discussion on this method and the inspiration for our proof.
\end{remark}

\appendix


\section{Appendix: Path-space large deviations for a projected process} \label{appendix:large_deviations_for_projected_processes}

As exhibited in the proof of Theorem~\ref{thm:subcrt:md:CWdiss} the convergence of Hamiltonians is a key step in the proof of the moderate deviation principle. In the setting of Theorems~\ref{thm:crt:md:CWdiss} and \ref{thm:tri-crt:md:CWdiss}, this method, in its naive application, runs into problems. Namely, the processes take their values in the space $\mathfrak{S} \subset \bR^2$, whereas our moderate deviation principle is only about the radial variable. As a consequence, the limiting Hamiltonian should only take into account this reduced description: we need to obtain path-space large deviations for a projected process.  

Technically speaking, an extended set-up is necessary to be able to talk about the convergence of operators. We use the general results in \cite{FK06} and explain how the large deviation principle can be proven when allowing for projected processes. The notation of the present section agrees with the notation in \cite{FK06}.\\
	
We need a notion of bounded and uniform convergence on compact sets (\emph{buc convergence}). First, we map our spaces $E_n$ into $\bR^+$ by maps $\eta_n : E_n \rightarrow \bR^+$. Then we cover our limiting state space $\bR^+$ with a collection of compact sets $K^q$, which are convenient to work with. Finally, we connect the aforementioned sets with compact sets $K_n^q$ in $E_n$, so that $K_n^q$ `converges' to $K^q$ in an appropriate sense. 

\smallskip 

In our application, the maps $\eta_n : \bR^2 \rightarrow \bR^+$ are given by $(\pi_1 \circ \Phi)(x,\xi) = x^2 + \xi^2$, with $\pi_1$ the projection on the first component and $\Phi$ the coordinate transformation in \eqref{definition:phi}. Define for every $q \geq 0$ the sets $K^q = [0,q]$ and $K_n^q = \eta_n^{-1}(K^q)$. 
Keeping the discussion below general, we make the following assumption.
	
\begin{assumption} \label{assumption:Appendix_convergence_of_spaces}
For each $q \in \bR^+$ there are compact sets $K^q \subseteq E$ and $K_n^q \subseteq E_n$. Moreover, we have
\begin{enumerate}[(a)]
\item For each compact set $K \subseteq \bR^+$, there exists a $q \in \mathbb{R}^+$ such that $K \subseteq K^q$.
\item If $q_1 \geq q_2$ then $K_n^{q_2} \subseteq K_n^{q_1}$ for all $n$.
\item It holds $\lim_{n \to \infty} \eta_n(K_q^n)  = K^q$; i.e., (1) for every $r \in K^q$, there exist $(x_n,\xi_n) \in K_n^q$ such that \mbox{$\eta_n(x_n,\xi_n) \rightarrow r$}, (2) for any increasing sequence $\{n(k)\}_{k \geq 1}$ and points $(x_{n(k)},\xi_{n(k)}) \in K^q_n$ such that $\lim_{k \to \infty} \eta_n(x_{n(k)},\xi_{n(k)}) = r$, we have $r \in K^q$.
\end{enumerate}
\end{assumption}

\begin{remark}
Our current set-up  is slightly easier than the corresponding set-up in \cite{CoKr18,CoGoKr18}, in the sense that the sets $\eta_n^{-1}(K_n^q)$ in those papers are non-compact. As an indirect consequence, here we do not have to work with upper and lower limiting operators $H_\dagger$ and $H_\ddagger$, which greatly simplifies the proof of the moderate deviation principles.
\end{remark}
	
In the next definition we give a notion of buc convergence. 

\begin{definition}[Def.~2.5 in \cite{FK06}]\label{def:definition_LIM}
Suppose Assumption~\ref{assumption:Appendix_convergence_of_spaces} holds true.	For $f_n \in C_b(E_n)$ and $f \in C_b(\mathbb{R}^+)$, we will write $\LIM f_n = f$ if, for all $q \geq 0$, we have $\sup_n \vn{f_n} < \infty$ and 
\begin{equation*}
\lim_{n \rightarrow \infty} \, \sup_{(x,y) \in K^q_n} \, \left|f_n(x,y) - f(\eta_n(x,y)) \right| = 0.
\end{equation*}
\end{definition}

We have now at our disposal the notions we need to define operator convergence in the case of projection.

\begin{definition}\label{def:definition_exLIM}
Suppose Assumption~\ref{assumption:Appendix_convergence_of_spaces} holds true. Moreover, assume that for each $n$ we have an operator $(B_n,\cD(B_n))$, $B_n :  C_b(E_n) \supseteq \cD(B_n) \rightarrow C_b(E_n)$. The \textit{extended limit} $ex-\LIM_n B_n$ is defined by the collection $(f,g) \in C_b(\mathbb{R}^+) \times C_b(\mathbb{R}^+)$ such that there exist $f_n \in \cD(B_n)$ satisfying
\begin{equation} \label{eqn:convergence_condition}
\lim_{n \rightarrow \infty} \sup_{(x,y) \in K_n^q} \left|f_n(x,y) - f(\eta_n(x,y))\right| + \left|B_n f_n(x,y) - g(\eta_n(x,y))\right| = 0,
\end{equation}
for all $q \geq 0$. For an operator $(B,\cD(B))$, we write $B \subseteq ex-\LIM_n B_n$ if the graph $\{(f,Bf) \, | \, f \in \cD(B) \}$ of $B$ is a subset of $ex-\LIM_n B_n$.
\end{definition}

We turn to the derivation of the large deviation principle. We first introduce our setting.

\begin{assumption} \label{assumption:LDP_assumption}
Suppose Assumption~\ref{assumption:Appendix_convergence_of_spaces} holds true. Assume that, for each $n \geq 1$, we have $A_n \subseteq C_b(E_n) \times C_b(E_n)$ and existence and uniqueness holds for the $D_{E_n}(\bR^+)$ martingale problem for $(A_n,\mu)$ for each initial distribution $\mu \in \cP(E_n)$. Letting $\PR_{y}^n \in \cP(D_{E_n}(\bR^+))$ be the solution to $(A_n,\delta_y)$, the mapping $y \mapsto \PR_y^n$ is measurable for the weak topology on $\cP(D_{E_n}(\bR^+))$. Let $X_n$ be the solution to the martingale problem for $A_n$ and set
\begin{equation*}
H_n f = \frac{e^{-r(n)f}A_n e^{r(n)f}}{r(n)},  \qquad e^{r(n)f} \in \cD(A_n),
\end{equation*}
for some sequence of speeds $\{r(n)\}_{n \geq 1}$, with $\lim_{n \rightarrow \infty} r(n) = \infty$. Moreover, suppose that we have an operator $H: C_b(\mathbb{R}^+) \supseteq \cD(H) \rightarrow C_b(\mathbb{R}^+)$ with $\cD(H) = C^\infty_c(\mathbb{R}^+)$ of the form $Hf(x) = H(x,\nabla f(x))$ which satisfies $H \subseteq ex-\LIM H_n$.
\end{assumption}

The convergence of Hamiltonians is a major component in the proof of the large deviation principles. A second important aspect is exponential tightness. Due to the convergence of the Hamiltonians, it suffices to establish an exponential compact containment condition that is suited to the particular structure of compact sets chosen for the convergence of functions, cf. \cite[Cor.~4.17]{FK06}.

\begin{definition}
	We say that a process $Z_n(t)$ on $E_n$ satisfies the exponential compact containment condition at speed $\{r(n)\}_{n \geq 1}$, with $\lim_{n \to \infty} r(n) = \infty$ if, for all $q \geq 0$,  constants $a \geq 0$, and times $T > 0$, there is a $q'= q'(q,a,T) \geq 0$ with the property that
	\begin{equation*}
	\limsup_{n \to \infty} \sup_{z \in K_n^q} \frac{1}{r(n)} \log \PR\left[Z_n(t) \notin K_n^{q'} \text{ for some } t \leq T \, \middle| \, Z_n(0) = z\right] \leq - a.
	\end{equation*}
\end{definition}

The exponential compact containment condition can be verified by using approximate Lyapunov functions and martingale methods. This is summarised in the following lemma. Note that exponential compact containment can be obtained by taking deterministic initial conditions.

\begin{lemma}[Lem. 4.22 in \cite{FK06}] \label{lemma:compact_containment_FK}
Suppose Assumption \ref{assumption:LDP_assumption} is satisfied. Let $Z_n(t)$ be solution of the martingale problem for $A_n$ and assume that $\{Z_n(0)\}_{n \geq 1}$ is exponentially tight with speed $\{r(n)\}_{n \geq 1}$. Let $q \geq 0$ and let $G \subseteq \mathbb{R}^2$ be open and such that for all $n$: $K_n^q := \eta_n^{-1}([0,q]) \subseteq G$. For each $n$, suppose we have $(f_n,g_n) \in H_n$. Define
\begin{align*}
\beta(q,G) &:= \liminf_{n \rightarrow \infty} \left( \inf_{(x,\xi) \in G^c \cap E_n} f_n(x,\xi) - \sup_{(x,\xi) \in K^q_n} f_n(x,\xi)\right), \\
\gamma(G) & := \limsup_{n \rightarrow \infty} \sup_{(x,\xi) \in G \cap E_n} g_n(x,\xi).
\end{align*}
Then
\begin{multline*}
\limsup_{n \rightarrow \infty} \frac{1}{r(n)} \log \PR\left[Z_n(t) \notin G \text{ for some } t \leq T  \right] \\
\leq \max \left\{-\beta(q,G) + T \gamma(G), \limsup_{n\rightarrow \infty} \PR\left[\eta_n(Z_n(0)) \notin [0,q] \right] \right\}.
\end{multline*}
\end{lemma}

The third ingredient for proving the large deviation principle is showing that the limiting Hamiltonian $H$ generates a semigroup. For this we combine the Crandall-Liggett generation theorem \cite{CL71} with the use of viscosity solutions. The main issue in the application of the Crandall-Liggett theorem is the verification of the so-called range condition: it is often hard to carry out. In \cite{FK06} the authors have introduced a method to extend the domain of the operator $H$ by the use of viscosity solutions to the associated Hamilton-Jacobi equation. By the definition of viscosity solutions, the extension satisfies by construction the conditions for the Crandall-Liggett theorem,  establishing the existence of a semigroup corresponding to $H$.

\begin{definition}[Viscosity solutions]
Let $H \subseteq C_b(\mathbb{R}^+) \times C_b(\mathbb{R}^+)$ and let $\lambda > 0$ and $h \in C_b(\mathbb{R}^+)$. Consider the Hamilton-Jacobi equation
\begin{equation} \label{eqn:differential_equation_intro}
f - \lambda H f = h.
\end{equation}
We say that $u$ is a \textit{(viscosity) subsolution} of equation \eqref{eqn:differential_equation_intro} if $u$ is bounded, upper semi-continuous and if, for every $f \in \cD(H)$ such that $\sup_x u(x) - f(x) < \infty$ and every sequence $x_n \in \mathbb{R}^+$ such that
\begin{equation*}
\lim_{n \rightarrow \infty} u(x_n) - f(x_n)  = \sup_x u(x) - f(x),
\end{equation*}
we have
\begin{equation*}
\lim_{n \rightarrow \infty} u(x_n) - \lambda Hf(x_n) - h(x_n) \leq 0.
\end{equation*}
We say that $v$ is a \textit{(viscosity) supersolution} of equation \eqref{eqn:differential_equation_intro} if $v$ is bounded, lower semi-continuous and if, for every $f \in \cD(H)$ such that $\inf_x v(x) - f(x) > - \infty$ and every sequence $x_n \in \mathbb{R}^+$ such that
\begin{equation*}
\lim_{n \rightarrow \infty} v(x_n) - f(x_n)  = \inf_x v(x) - f(x),
\end{equation*}
we have
\begin{equation*}
\lim_{n \rightarrow \infty} v(x_n) - \lambda Hf(x_n) - h(x_n) \geq 0.
\end{equation*}
We say that $u$ is a \textit{(viscosity) solution} of equation \eqref{eqn:differential_equation_intro} if it is both a subsolution and a supersolution.

We say that \eqref{eqn:differential_equation_intro} satisfies the \textit{comparison principle} if for every subsolution $u$ and supersolution $v$, we have $u \leq v$.
\end{definition}

Note that the comparison principle implies uniqueness of viscosity solutions. This in turn implies that a unique extension of the Hamiltonian can be constructed based on the set of viscosity solutions.

To conclude we state the main result of this Appendix: the large deviation principle. 


\begin{theorem}[Large deviation principle] \label{theorem:Abstract_LDP}
Suppose we are in the setting of Assumption \ref{assumption:LDP_assumption} and assume that $\Upsilon$ is a good containment function for $H$. Let $Z_n(t)$ be the solution to the martingale problem for $A_n$. 

Suppose that the large deviation principle at speed $\{r(n)\}_{n \geq 1}$ holds for $\eta_n(Z_n(0))$ on the space $\bR^+$ with good rate-function $I_0$. Additionally suppose that the exponential compact containment condition holds at speed $\{r(n)\}_{n \geq 1}$ for the processes $Z_n(t)$.

Finally, suppose that for all $\lambda > 0$ and $h \in C_b(\mathbb{R}^+)$ the comparison principle holds for \mbox{$f - \lambda H f = h$}.  

Then the large deviation principle holds with speed $\{r(n)\}_{n \geq 1}$  for $\{\eta_n(Z_n(t))\}_{n \geq 1}$ on $D_{\mathbb{R}^+}(\bR^+)$ with good rate function $I$. Additionally, suppose that the map $p \mapsto H(x,p)$ is convex and differentiable for every $x$ and that the map $(x,p) \mapsto \frac{\dd}{\dd p} H(x,p)$ is continuous. Then the rate function $I$ is given by
\begin{equation*}
I(\gamma) = \begin{cases}
I_0(\gamma(0)) + \int_0^\infty \cL(\gamma(s),\dot{\gamma}(s)) \dd s & \text{if } \gamma \in \cA\cC, \\
\infty & \text{otherwise},
\end{cases}
\end{equation*}
where $\cL : \mathbb{R}^+ \times \bR \rightarrow \bR$ is defined by $\cL(x,v) = \sup_p pv - H(x,p)$.
\end{theorem}

\begin{proof}
The large deviation result follows by \cite[Thm.~2.10]{FK06}. This result is a special case of the more general \cite[Thm.~7.18]{FK06}, which can be applied with $H_\dagger = H_\ddagger$ equal to our $H$, $F = C_c^\infty(\bR^+)$ and $S = \bR$.
	
The variational representation of the rate function follows from \cite[Cor.~8.28]{FK06} with $\mathbf{H} = H$. The verification of the conditions for Corollary 8.28 corresponding to a Hamiltonian of this type have been carried out in e.g. \cite[Sect.~10.3]{FK06} or in \cite{CoKr17}.
\end{proof}

\section{Appendix: Bifurcation analysis for the infinite volume dynamics}
\label{appendix:proof_thm_phase_diagram}

\paragraph{Proof of Theorem~\ref{thm:phase_diagram}.} We first discuss existence and local stability of equilibria for the dynamical system \eqref{CWdiss:macro:dyn}. Secondly, we investigate the possibility of having other types of attractors.\\

\emph{Existence and local stability of stationary solutions.} For all the values of the parameters $\beta$ and $\kappa$, the dynamical system \eqref{CWdiss:macro:dyn} admits $(0,0)$ as unique fixed point. For $\beta < \frac{\kappa + 2 \Gamma(0)}{2 \Gamma'(0)}$ the origin is locally stable; whereas, for $\beta > \frac{\kappa + 2 \Gamma(0)}{2 \Gamma'(0)}$, it loses its stability. On the critical line $\beta  = \frac{\kappa + 2 \Gamma(0)}{2 \Gamma'(0)}$, a Hopf bifurcation occurs.\\

\emph{Hopf bifurcation and tri-critical point.} Set $\beta = \frac{\kappa + 2 \Gamma(0)}{2 \Gamma'(0)}$. To determine whether the occurring Hopf bifurcation is super- or subcritical we compute the Lyapunov number $\sigma_{\mathrm{L}}$ associated with the focus at the origin. We use formula $(3')$ in \cite[Sect.~4.4]{Per01}. It yields
\[
\sigma_{\mathrm{L}} = - \frac{3 \pi \, \Gamma(0) \big( \kappa + 2 \Gamma(0) \big) \big(6\Gamma''(0) \Gamma'(0) - (\kappa + 2 \Gamma(0)) \Gamma'''(0) \big)}{8 \sqrt{2 \kappa \Gamma(0)} \big( \Gamma'(0) \big)^3}. 
\]
By \cite[Thm.~1, Sect.~4.4]{Per01} the bifurcation is supercritical whenever $\sigma_{\mathrm{L}} < 0$ and subcritical if, on the contrary, $\sigma_{\mathrm{L}} > 0$, giving possible existence of the tri-critical point $(\kappa_{\mathrm{tc}},\beta_{\mathrm{tc}})$.\\ 
If $\sigma_{\mathrm{L}} = 0$, to determine the type of Hopf bifurcation, we need to compute the next-order Lyapunov number, whose sign is given by $- \mathrm{sgn} \big[ 5 \Gamma^{(4)}(0) \Gamma'''(0) - 3 \Gamma^{(5)}(0) \Gamma''(0) \big]$ (cf. the Lagrangian in Theorem~\ref{thm:tri-crt:md:CWdiss}). If the latter expression vanishes, we can go even further and repeat the same reasoning. We will not do it and we refer to \cite[pp. 218-219]{Per01} for additional details. \\ 

As far as moderate deviations are concerned, the study of the local stability of the origin would suffice, since we are able to characterize only fluctuations around the fixed point. Nevertheless, to get a complete understanding of their behavior it is useful to derive global properties of the attractors of \eqref{CWdiss:macro:dyn}.\\

\emph{Global analysis.} It is useful to perform a change of variables and cast the dynamical system \eqref{CWdiss:macro:dyn} in Li\'enard form (cf. equation $(1)$ in \cite[Sect.~3.8]{Per01}), that allows for a detailed study of global stability. We consider the transformation $x = \zeta - \beta m$, $\xi = I(\zeta) := \int_{0}^{\zeta} \left( \Gamma(u) + \Gamma(-u) \right)^{-1} du$. Observe that this transformation does not shift the fixed point and, moreover, it is invertible as $\zeta \mapsto I(\zeta)$ is a strictly increasing mapping on $\mathbb{R}$. In the new variables $(x,\xi)$, system \eqref{CWdiss:macro:dyn} becomes
\begin{equation}\label{eqn:Lienard_system}
\left\{
\begin{array}{l}
\dot{x}(t) = - \kappa \, I^{-1} \left( \xi(t) \right) \\[.2cm]
\dot{\xi}(t) = x -  \left( \mathbb{F}_{\beta,\kappa} [\Gamma] \circ I^{-1} \right) \left( \xi(t) \right),
\end{array}
\right.
\end{equation}
with
\begin{equation}\label{def:Lienard_function_F}
\mathbb{F}_{\beta,\kappa}[\Gamma](u) = \frac{\left( \Gamma(u) + \Gamma(-u) + \kappa \right) u - \beta \left( \Gamma(u) - \Gamma(-u) \right)}{\Gamma(u) + \Gamma(-u)}.
\end{equation}
The number of periodic solutions for the Li\'enard system \eqref{eqn:Lienard_system} is related to the number of zeroes of the function $\mathbb{F}_{\beta,\kappa} [\Gamma] \circ I^{-1}$, see \cite{Oda96}. Observe that the function $I^{-1}$ is odd and strictly increasing. Therefore, $\mathbb{F}_{\beta,\kappa}[\Gamma] \circ I^{-1}$ is also odd, as composition of odd functions, and moreover it inherits the intervals of increase/decrease and the number of zeroes from $\mathbb{F}_{\beta,\kappa}[\Gamma]$. A quick analysis of the zeroes of $\mathbb{F}_{\beta,\kappa}[\Gamma]$ is carried out in the next paragraph. It is postponed for the sake of readability.\\
We proceed now with the analysis of the attractors and thus with the proof of the statements in Theorem~\ref{thm:phase_diagram}. \\

Setting (I), case $\beta \leq \beta_{\mathrm{c}}(\kappa)$. In this parameter range, due to F2 below, we have that $u \mathbb{F}_{\beta,\kappa}[\Gamma](u) > 0$ for every $u \neq 0$. As a consequence, the function
\begin{equation}\label{eqn:Lyapunov_function}
W(x,\xi) = \frac{x^2}{2} + \kappa \int_{0}^{\xi} I^{-1}(u)du
\end{equation}
is a global Lyapunov function implying global stability of the origin.\\

Setting (I), case $\beta > \beta_{\mathrm{c}}(\kappa)$. First of all note that $\xi I^{-1}(\xi) > 0$, for all $\xi \neq 0$, and that 
\[
\left(\mathbb{F}_{\beta,\kappa}'[\Gamma] \circ I^{-1} \right)(0) = \frac{\kappa + 2 \Gamma(0) - \beta \Gamma'(0)}{4 \Gamma^2(0)} < 0,
\]
since we are supercritical. Moreover, due to F1 below, the function $\mathbb{F}_{\beta,\kappa}[\Gamma] \circ I^{-1}$ has exactly one positive zero at $\xi=\xi_*$ and $\mathbb{F}_{\beta,\kappa}[\Gamma] \circ I^{-1}$ is monotonically increasing to infinity for $\xi > \xi_*$. Therefore, standard Li\'enard’s theorem guarantees existence and uniqueness of a stable periodic orbit (see \cite[Thm.~1,Sect.~3.8]{Per01}).\\
 
Scenario (IIA). The statement is proven analogously to scenario (I).\\

Scenario  (IIB). If $\kappa > \kappa_{\mathrm{tc}}$ the dependence of the attractors on the parameter $\beta$ is quite nontrivial, due to the occurrence of both a \emph{subcritical} Hopf bifurcation and a saddle-node bifurcation of periodic orbits. To ease the readability of the remaining part of the proof, we first explain what is happening and then we give the technical details.

Roughly speaking, there are three possible phases for system \eqref{CWdiss:macro:dyn}:
\begin{itemize}
\item
Fixed point phase. For $\beta < \beta_\star$ the only stable attractor is $(0,0)$. 
\item
Coexistence phase. For $\beta = \beta_\star$ the system has a \emph{semistable} cycle surrounding the origin. By increasing the parameter $\beta$ from $\beta_\star$, this cycle splits into two limit cycles, the outer being stable and the inner unstable. In this phase $(0, 0)$ is linearly stable. Therefore, the locally stable fixed point coexists with a stable periodic orbit. 
\item
Periodic orbit phase. For $\beta = \frac{\kappa + 2 \Gamma(0)}{2\Gamma'(0)}$ the Hopf bifurcation occurs: the inner unstable limit cycle disappears collapsing at the origin. At the same time the equilibrium loses its stability and, thus, the external stable limit cycle remains the only stable attractor for $\beta \geq \frac{\kappa + 2 \Gamma(0)}{2\Gamma'(0)}$.
\end{itemize}
Now let us go into the details of the proof.

Due to F1 and F4 below, the assertion (IIB3) follows by applying Li\'enard theorem as in (IB).\\

The key point for proving (IIB2) is the particular structure of the vector field generated by \eqref{eqn:Lienard_system}: it defines a semicomplete one-parameter family positively \emph{rotated} vector fields (with respect to $\beta$, for fixed $\kappa$), see \cite[Def.~1, Sect.~4.6]{Per01}. For dynamical systems that depend on this specific way on a parameter, many results concerning bifurcations, stability and global behavior of limit cycles and separatrix cycles are known \cite[Chap.~4]{Per01}. In particular, we will exploit the following properties: 
\begin{enumerate}
\item[a.] 
limit cycles expand/contract monotonically as the parameter $\beta$ varies in a fixed sense;
\item[b.]
a limit cycle is generated/absorbed either by a critical point or by a separatrix of \eqref{eqn:Lienard_system};
\item[c.] 
cycles of distinct fields do not intersect.
\end{enumerate}
Properties a and b allow to explain the appearance of a separatrix cycle whose breakdown causes a saddle-node bifurcation of periodic orbits at $\beta=\beta_{\star}$. Recall we proved that for $\beta \geq \beta_{\mathrm{c}}(\kappa)$ a stable limit cycle exists.  While decreasing $\beta$ from $\beta_{\mathrm{c}}(\kappa)$ this cycle shrinks and, at the same time, the periodic orbit arisen at the Hopf point expands, until they collide forming the semistable cycle at $\beta=\beta_{\star}$. Looking the same ``process'' forwardly, we can understand what is going on in this phase. When the separatrix splits increasing $\beta$ from $\beta_{\star}$, it generates two periodic orbits%
\footnote{The existence of two periodic orbits is indeed consistent with property F3(b) below. Observe moreover that $\beta_{\Delta}$ is a lower bound for $\beta_{\mathrm{\star}}$, as the limit cycles exist whenever the local extrema of the function $\mathbb{F}_{\beta,\kappa}$ reach a proper relative depth/height (precise conditions can be found in \cite{Oda96}).}  
both surrounding $(0,0)$. The inner limit cycle is unstable (due to the subcritical Hopf bifurcation at $\beta=\beta_{\mathrm{c}}(\kappa)$) and represents the boundary of the basin of attraction of $(0, 0)$. Moreover, the external periodic orbit inherits the stability of the exterior of semistable cycle and so it is stable. See \cite[Thm.~2 and Fig.~1, Sect.~4.6]{Per01} for more details.\\

Now we prove (IIB1). Notice that the total derivative of the Lyapunov function \eqref{eqn:Lyapunov_function} is negative for every $x \in \mathbb{R}$ and for every $\xi \geq I(\beta)$. Therefore, there exists a stable domain for the flux of \eqref{eqn:Lienard_system} and, in particular, the trajectories can not escape to infinity as $t \to +\infty$.  To conclude it suffices to prove that in this phase the dynamical system \eqref{eqn:Lienard_system} does not admit a periodic solution. Indeed, the non-existence of cycles together with the existence of a stable domain for the flux guarantee that every trajectory must converge to an equilibrium as $t \to +\infty$.  \\
Thus, we are left with showing that no periodic orbit exists for $\beta < \beta_{\star}(\kappa)$. From properties a and b it follows that, as $\beta$ increases from $\beta_{\star}$ to infinity, the outer stable limit cycle expands and its motion covers the whole region external to the separatrix. Similarly, the inner unstable cycle contracts from it and terminates at the critical points $(0,0)$. As a consequence, for $\beta >\beta_{\star}$ the entire phase space is covered by expanding or contracting limit cycles. Now, by using property c, we can deduce that  no periodic trajectory may exist  for $\beta < \beta_{\star}$. In fact, such an orbit would intersect some of the cycles present when  $\beta>\beta_{\star}$ that is not possible. \\

This concludes the proof of Theorem~\ref{thm:phase_diagram}. We end the present section by studying the zeroes of function $\mathbb{F}_{\beta,\kappa}[\Gamma]$.

\paragraph{Zeroes of $\mathbb{F}_{\beta,\kappa}[\Gamma]$.} We are interested in controlling the number of zeros of the function $\mathbb{F}_{\beta, \kappa}[\Gamma]$ given in \eqref{def:Lienard_function_F}. Equivalently, we look for solutions of the fixed point equation
\begin{equation}\label{eqn:fixed_point_equation}
u = \Xi_{\beta,\kappa}[\Gamma](u) \quad \text{ with } \quad  \Xi_{\beta,\kappa}[\Gamma](u) := \frac{\beta\big( \Gamma(u)-\Gamma(-u) \big)}{\Gamma(u)+\Gamma(-u)+\kappa}.
\end{equation}
Note that $u=0$ is a solution of \eqref{eqn:fixed_point_equation} for all $\beta$ and $\kappa$. We investigate under what conditions we can have non-zero solutions. Due to the oddness of both sides in \eqref{eqn:fixed_point_equation}, we restrict our analysis on $[0,+\infty)$. \\
The function $\Xi_{\beta,\kappa}[\Gamma]$ is continuous, positive and strictly increasing on $[0,+\infty)$ for all values of the parameters. Moreover, we get $\lim_{u \to +\infty} \Xi_{\beta,\kappa}[\Gamma](u) = \ell$ ($\leq \beta$). Therefore, in general, if
\begin{equation}\label{relation:critical_value}
\Xi_{\beta,\kappa}'[\Gamma](0) = \frac{2 \beta \Gamma'(0)}{\kappa + 2 \Gamma(0)} > 1,
\end{equation}
there may be at least one positive solution. Observe that $\Xi_{\beta,\kappa}'[\Gamma](0) > 1$ (resp. $\Xi_{\beta,\kappa}'[\Gamma](0)=1$ or $\Xi_{\beta,\kappa}'[\Gamma](0)<1$) corresponds to $\beta > \beta_{\mathrm{c}}(\kappa)$ (resp. $\beta = \beta_{\mathrm{c}}(\kappa)$ or $\beta < \beta_{\mathrm{c}}(\kappa)$).\\
Since $\Xi_{\beta,\kappa}[\Gamma]$  is not always concave, there may be a positive solution even when \eqref{relation:critical_value} fails. For this reason, we analyze the curvature  of $\Xi_{\beta,\kappa}[\Gamma]$. We base our study on the sign of $\Xi_{\beta,\kappa}'''[\Gamma](0)$, as $\Xi_{\beta,\kappa}''[\Gamma](0) \equiv 0$ for all the values of the parameters.\\ 
By Assumption~\ref{assumption_Gamma}(iii), $\Xi_{\beta,\kappa}[\Gamma]$ may have at most one inflection point. As a consequence, $\Xi_{\beta,\kappa}[\Gamma]$ changes curvature at most once. We can argue as follows. Since as $u \to +\infty$ the function $\Xi_{\beta,\kappa}[\Gamma]$ approaches its limit $\ell$ from below, it must be concave for large $u$. Then, 
\begin{enumerate}[F1.]
\item 
if $\beta > \beta_{\mathrm{c}}(\kappa)$, no matter if either $\Xi_{\beta,\kappa}'''[\Gamma](0) < 0$ or $\Xi_{\beta,\kappa}'''[\Gamma](0) > 0$, the curve $\Xi_{\beta,\kappa}[\Gamma](u)$ crosses the diagonal at precisely one positive $u$.
\item 
if $\beta \leq \beta_{\mathrm{c}}(\kappa)$ and $\Xi_{\beta,\kappa}'''[\Gamma](0) < 0$, then $\Xi_{\beta,\kappa}[\Gamma](u)$ must be strictly concave on $[0,+\infty)$ for all values of the parameters and hence there is no intersection with the diagonal.
\item  
if $\beta < \beta_{\mathrm{c}}(\kappa)$ and $\Xi_{\beta,\kappa}'''[\Gamma](0) > 0$, $\Xi_{\beta,\kappa}[\Gamma](u)$ changes curvature either below or above the diagonal, giving rise to none or two positive fixed points. As the mapping $\beta \mapsto \Xi_{\beta, \kappa}[\Gamma]$ is strictly increasing, we get two well-defined regions separated by a boundary curve $\beta_{\Delta}(\kappa) \leq \beta_{\mathrm{c}}(\kappa)$, that corresponds to the choice of parameters where there exists $u_* > 0$ such that $\Xi_{\beta,\kappa}[\Gamma](u_*) = u_*$ and $\Xi_{\beta,\kappa}[\Gamma](u_*) = 1$. More precisely, we have
\begin{enumerate}[(a)]
\item 
for $\beta < \beta_{\Delta}(\kappa)$ there is no intersection with the diagonal; 
\item 
for $\beta_{\Delta}(\kappa) < \beta < \beta_{\mathrm{c}}(\kappa)$, the diagonal and $\Xi_{\beta,\kappa}[\Gamma]$ intersect two times (corresponding to the curve crossing the diagonal first from below and then from above).
\end{enumerate}
\item 
if $\beta = \beta_{\mathrm{c}}(\kappa)$ and $\Xi_{\beta,\kappa}'''[\Gamma](0) > 0$, there is exactly one positive solution of \eqref{eqn:fixed_point_equation}.
\end{enumerate}
Observe, in particular, that
\[
\Xi_{\beta,\kappa}'''[\Gamma](0) = - \frac{2 \beta \big( 6 \Gamma''(0) \Gamma'(0) - (\kappa + 2\Gamma(0)) \Gamma'''(0) \big)}{\left( \kappa + 2 \Gamma(0) \right)^2}.
\]

\medskip

\textbf{Acknowledgments}
The authors wish to thank Marco Formentin for fruitful discussions regarding the content of Theorem~\ref{thm:phase_diagram}. FC was supported by The Netherlands Organisation for Scientific Research (NWO) via TOP-1 grant 613.001.552. RCK was partially supported by the Deutsche Forschungsgemeinschaft (DFG) via RTG 2131 High-dimensional Phenomena in Probability – Fluctuations and Discontinuity.

\bibliographystyle{abbrv} 
\bibliography{../KraaijBib,../ColletBib}{}

\end{document}